\documentclass{amsart}
\usepackage{amssymb}
\usepackage[colorlinks,citecolor=blue,linkcolor=red]{hyperref}

\theoremstyle{plain}
\newtheorem{theorem}{Theorem}[section]

\newtheorem{proposition}[theorem]{Proposition}
\newtheorem{lemma}[theorem]{Lemma}

\theoremstyle{remark}

\newtheorem{remark}[theorem]{Remark}

\theoremstyle{definition}
\newtheorem{definition}[theorem]{Definition}

\numberwithin{equation}{section} 
\numberwithin{figure}{section}

\newcommand{\abs}[1]{\lvert #1 \rvert}
\newcommand{\labs}[1]{\left\lvert #1 \right\rvert}
\newcommand{\norm}[1]{\lVert #1 \rVert}
\newcommand{\scalar}[2]{\langle #1 \mid #2 \rangle}

\newcommand{\neh}{\mathcal{M}}
\newcommand{\settc}[2]{\left\{\,#1 \Bigm\vert #2\,\right\}}

\usepackage{ulem}

\newcommand{\ua}{u_{c}}
\newcommand{\uai}{u_{\infty}}
\newcommand{\fa}{f}
\newcommand{\Ga}{G}
\newcommand{\La}{\mathcal{L}_{\lambda}}

\newcommand{\ue}{w_{c}}
\newcommand{\uei}{w_{\infty}}
\newcommand{\fe}{g}
\newcommand{\Ge}{R}
\newcommand{\Le}{\mathcal{L}_{\omega_{\infty}}}

\begin{document}

\title[Asymptotic properties of non-relativistic limit]{Asymptotic
properties of non-relativistic limit for pseudo-relativistic Hartree
equations}

\author{Pan Chen} \address[Chen]{School of Mathematical Sciences, Shanghai
Jiao Tong University \\ Shanghai 200240, P.R. China}

\author{Vittorio Coti Zelati} \address[Coti Zelati]{Dipartimento di
Matematica Pura e Applicata ``R. Caccioppoli'' \\ Universit\`a di
Napoli ``Federico II'', M. S. Angelo, 80126 Napoli, Italy}

\author{Yuanhong Wei} \address[Wei]{School of Mathematics, Jilin
University \\ Changchun 130012, P. R. China}

\begin{abstract}
	In this paper, we study the asymptotic behavior of energy and
	action ground states to the following pseudo-relativistic Hartree
	equation 
	\begin{equation*}
		\left(\sqrt{-c^2\Delta +m^2c^4}-mc^2\right)u + \lambda u =
		\left(|x|^{-1}*|u|^2\right)u
	\end{equation*}
	as the speed of light $c\to \infty$.
	
	We obtain an asymptotic expansion of the ground state as $c \to
	\infty,$ which is new in the case of the energy ground state and
	generalizes the results of Choi, Hong, and Seok \cite{MR3734982}
	for the action ground state.
\end{abstract}

\keywords{Non-relativistic limit, Pseudo-relativistic operator, Ground state}

\subjclass{35Q40  35J50 49J35}

\thanks{The first and second author thanks the third author for the
warm hospitality during their visit to the School of Mathematics of
the Jilin University}

\thanks{The third author was supported by the National Natural Science
Foundation of China (Grant No. 12571120).}

\maketitle
\setcounter{tocdepth}{1}
\tableofcontents

\section{Introduction}
We study the asymptotic behaviour as $c \to +\infty$ of the
pseudo-relativistic Hartree equation
\begin{equation}
	\label{1.1}
	i\partial_t\psi = (\sqrt{-c^2\Delta+m^2c^4}-mc^2)\psi-
	\mathcal{N}(\psi), \quad \text{ on } \mathbb{R}^3.
\end{equation}
where $\psi=\psi(t,x)$ is a complex-valued wave function,
$\mathcal{N}(\psi)=(|x|^{-1}*\abs{\psi}^2)\psi$ is a
nonlinearity of Hartree type, the symbol $*$ stands for convolution,
$c > 0$ represents the speed of light and $m > 0$ is the particle
mass.

The operator $\sqrt{-c^2\Delta+m^2c^4}$ is the pseudo-relativistic
operator which can be defined for $u\in H^1(\mathbb{R}^3)$ as the
inverse Fourier transform of the function
\begin{equation*}
	\sqrt{c^2|\xi|^2+m^2c^4}\mathcal{F}[u](\xi)
\end{equation*}
(here $\mathcal{F}[u] $ represents the Fourier transform of $u$), that
is
\begin{equation*}
	\sqrt{-c^2\Delta+m^2c^4} \, u(x)= \mathcal{F}^{-1}
	\left[\sqrt{c^2|\xi|^2 + m^2c^4} \mathcal{F}[u](\xi) \right](x),
	\qquad \forall u \in H^{1}(\mathbb{R}^3).
\end{equation*}

Here and in the following we work in Sobolev spaces
$H^{s}(\mathbb{R}^{3})$ for $s \geq 0$ defined as
\begin{equation*}
	H^{s}(\mathbb{R}^{3}) = \left\{ u \in L^{2}(\mathbb{R}^{3}) \;:\;
	(1 + |\xi|^2)^{s/2} \mathcal{F}[u](\xi) \in L^{2}(\mathbb{R}^{3})
	\right\}.
\end{equation*}

As we have done above, where we have defined the operator
$\sqrt{-c^{2}\Delta + m^{2}c^{4}}$ via Fourier transform and its
symbol $\sqrt{c^{2}\abs{\xi}^{2} + m^{2}c^{4}}$, we will identify the
differential operator $P(D)$ with its symbol $P(\xi)$ also when
$P(\xi)$ is a polynomial. In all case we have that
\begin{equation*}
	(P(D)[u])(x) = \mathcal{F}^{-1} \left[P(\xi) \mathcal{F}[u](\xi)
	\right](x)
\end{equation*}
for all $u \in H^{s}$ provided $P(\xi)(1+\abs{\xi}^{2})^{-s/2}$ is
bounded. In such a case $P(D)$ maps $H^{s+t}$ to $H^{t}$ for all $t
\geq 0$. For more information on pseudo-differential operator we refer
to \cite{Taylor_2023}.

Physically, equation \eqref{1.1} is an important model to describe the
dynamics of pseudo-relativistic boson stars in the mean field limit.
The equation \eqref{1.1} formally converges to the non-relativistic
Hartree equation
\begin{equation*}
	i\partial_t\psi=-\frac{\Delta}{2m}\psi- \mathcal{N}(\psi), \,\,
	\text{ on} \,\,\mathbb{R}^3.
\end{equation*}
as $c\to \infty$, since $\sqrt{c^2|\xi|^2+m^2c^4}-mc^2 \to
\frac{|\xi|^2}{2m}$ as $c\to \infty$.

Looking for a stationary solution, inserting $\psi(t,x)= e^{i\lambda
t}u(x)$ with $\lambda>0$, we obtain the time-independent equation
\begin{equation}\label{1.2}
    \left(\sqrt{-c^2\Delta+m^2c^4}-mc^2\right)u+\lambda u= \mathcal{N}(u),
\end{equation}
which can be considered as the relativistic versions of usual
nonlinear Hartree equation
\begin{equation}\label{1.3}
    -\frac{1}{2m}\Delta u+\lambda u= \mathcal{N}(u).
\end{equation}
The existence of positive solutions to \eqref{1.2} can be addressed by
variational methods in two different ways, either by minimizing the
action functional $\mathcal{J}_c \colon H^{1/2}(\mathbb{R}^3) \to
\mathbb{R}$:
\begin{equation*}
	\mathcal{J}_c(u):= \int_{\mathbb{R}^3}u(\sqrt{-c^2\Delta
	+m^2c^4}u-mc^2u + \lambda u) \, dx -\frac{1}{2}
	\int_{\mathbb{R}^3\times \mathbb{R}^3} \frac{u^2(x)u^2(y)}{|x-y|}
	\, d x d y
\end{equation*}
on the associated Nehari manifold
\begin{equation}
	\neh_c:= \settc{u\in H^{1/2}(\mathbb{R}^3) \setminus \{0\}}{
	d\mathcal{J}_c(u)[u] = 0 },
\end{equation}
or by minimizing the energy functional $\mathcal{E}_c \colon
H^{1/2}(\mathbb{R}^3)\to \mathbb{R}$
\begin{equation}
	\mathcal{E}_c(u):= \int_{\mathbb{R}^3}u(\sqrt{-c^2\Delta
	+m^2c^4}u-mc^2u)dx -\frac{1}{2} \int_{\mathbb{R}^3\times
	\mathbb{R}^3} \frac{u^2(x)u^2(y)}{|x-y|}d x d y
\end{equation}
on the $L^2$ sphere
\begin{equation*}
	\mathcal{S}= \settc {u\in H^{1/2}(\mathbb{R}^3)} {\|u\|_{L^2}=1}.
\end{equation*}
Similarly, in the non-relativistic case, one can find the solution of
\eqref{1.3} by minimizing the action functional $\mathcal{J}_\infty:
H^{1}(\mathbb{R}^3)\to \mathbb{R}$:
\begin{equation*}
	\mathcal{J}_\infty(u):= \frac{1}{2m}\int_{\mathbb{R}^3} |\nabla
	u|^2 + \lambda\int_{\mathbb{R}^3} | u|^2 -\frac{1}{2}
	\int_{\mathbb{R}^3\times \mathbb{R}^3} \frac{u^2(x)u^2(y)}{|x-y|}
	\, d x d y
\end{equation*}
on the manifold
\begin{equation}
	\neh_\infty := \settc {u\in H^{1}(\mathbb{R}^3) \setminus \{0\}}
	{d\mathcal{J}_\infty(u)[u] = 0} ,
\end{equation}
or by minimizing the energy functional
$\mathcal{E}_\infty : H^{1}(\mathbb{R}^3)\to \mathbb{R}$
\begin{equation}
	\mathcal{E}_\infty(u):= \frac{1}{2m}\int_{\mathbb{R}^3} |\nabla
	u|^2 -\frac{1}{2} \int_{\mathbb{R}^3\times \mathbb{R}^3}
	\frac{u^2(x)u^2(y)}{|x-y|} \, d x d y
\end{equation}
on the $L^2$ sphere
\begin{equation*}
	\mathcal{S}' = \settc {u\in H^{1}(\mathbb{R}^3)} {\|u\|_{L^2} =
	1}.
\end{equation*}

\begin{definition}[Ground states]
    Using the notations introduced above,
    \begin{enumerate}
		\item (Action ground state) A function $\ua \in \neh_c$ is
		called a (pseudo-relativistic) action ground state of
		\eqref{1.2} if
		\begin{equation*}
			\mathcal{J}_c(\ua)= \inf_{u \in \neh_c}\mathcal{J}_c(u).
		\end{equation*}
		Respectively, $\uai \in \neh_\infty$ is called a
		(non-relativistic) action ground state of \eqref{1.3} if
		\begin{equation*}
			\mathcal{J}_\infty(\uai) = \inf_{u \in
			\neh_\infty}\mathcal{J}_\infty(u).
		\end{equation*}

		\item (Energy ground state) A function $\ue \in \mathcal{S}$ is
		called an (pseudo-relativistic) energy ground state (or
		minimizers) for the variational problem
		\begin{equation}
			\label{1.8}
			e_c = \inf_{u \in \mathcal{S}} \mathcal{E}_c(u),
		\end{equation}
		if $\mathcal{E}_c(\ue)= e_c$.
		
		Respectively, $\uei \in \mathcal{S}'$ is called a
		(non-relativistic) energy ground state (or minimizers) for the
		variational problem
		\begin{equation}
			\label{1.9}
			e_\infty=\inf_{u\in \mathcal{S}'}\mathcal{E}_\infty(u)
		\end{equation}
		provided $\mathcal{E}_\infty(\uei)= e_\infty$.
    \end{enumerate}
\end{definition}

It is easy to see that the energy ground state $\ue$, $\uei$ for
problem \eqref{1.8} and \eqref{1.9} satisfies the Hartree equation,
\begin{align}
	\label{1.10}
	&\left(\sqrt{-c^2\Delta+m^2c^4}-mc^2\right) \ue +\omega_c \ue =
	\mathcal{N}(\ue), && \int_{\mathbb{R}^3}|\ue|^2 =1,
\end{align}
and
\begin{align}\label{1.11}
	&-\frac{1}{2m}\Delta \uei + \omega_\infty \uei =
	\mathcal{N}(\uei), && \int_{\mathbb{R}^3}|\uei|^2 =1,
\end{align}
with Lagrange multiplier $\omega_c$, $\omega_\infty$ respectively. 

These problems have received a lot of attention. For the
non-relativistic Hartree equations, results go back to the work of
Lieb \cite{zbMATH03576003} and PL Lions \cite{zbMATH03859846}, see
\cite{zbMATH06706641} for a comprehensive discussion of the problem.
Let us remark that the non-relativistic action and energy problems are
equivalent. One can go from solutions of one problem to solutions of
the other via multiplication and rescaling (see for example
\cite[Proposition 2.1]{MR3056699}). This is no more true for the
pseudo-relativistic action and energy, due to the behaviour of the
linear operator $\sqrt{-c^2\Delta +m^2c^4}-mc^2$ under rescaling.

Over the past few years, many works have been devoted also to the
mathematical analysis of problem \eqref{1.10}, see \cite{MR904142,
MR2318846, MR2561169,MR4078531,MR4430585,MR3788864,MR3712013}. For the
reader's convenience, we provide a concise summary of the existence
and uniqueness results concerning ground states, as well as a list of
their fundamental properties.

\begin{theorem}[Existence and properties of energy ground state]
	\label{thm:energyGS}
    There exists $c_0 > 0$ , such that the following holds.
    \begin{enumerate}
		\item (Existence and uniqueness of ground state) For $c >
		c_0$, up to phase and translation, there exists a unique
		positive radial ground state $\ue \in H^{1}(\mathbb{R}^3)$ for
		problem \eqref{1.8}. Moreover, $\ue$ satisfies the
		pseudo-relativistic Hartree equation \eqref{1.10} with some
		Lagrange multiplier $\omega_c > 0$.

		\item (Regularity) Let $\ue$ be the energy ground state, then
		$\ue$ belongs to $H^s(\mathbb{R}^3)$ for each $s>0$.
		
		\item (Non-relativistic limit) Let $\ue$ ($\uei$) be the
		unique positive radial energy ground state of the variational
		problem \eqref{1.8} (resp., \eqref{1.9}), then $\ue$ ($\uei$)
		satisfy the equation \eqref{1.10} (resp., \eqref{1.11}) with
		Lagrange multiplier $\omega_c > 0$ ($\omega_\infty > 0$).
		Moreover,
		\begin{align*}
			&\ue \to \uei \quad \text{in} \quad
			H^1(\mathbb{R}^3),\quad \text{as} \quad c\to \infty,\\
			&\omega_c \to \omega_\infty, \quad e_c\to e_\infty,\quad
			\text{as} \quad c\to \infty.
		\end{align*}
	\end{enumerate}
\end{theorem}

\begin{remark}
	The existence and regularity of a positive radial energy ground
	state goes back to Lieb and Yau \cite{MR904142} while uniqueness
	was first proved by Lenzmann in \cite{MR2561169} for almost all
	values of $c > c_{0}$ and then for all values of $c > c_{0}$ in
	\cite{MR4078531} by Guo and Zeng.
	
	The results about the non-relativistic limit are due to Lenzmann in
	\cite{MR2561169}.
\end{remark}

Similar results hold for the action ground states:
\begin{theorem}[Existence and properties of action ground states]
	\label{them:actionGS}
    There exists $c_0 > 0$ , such that the following holds.
    \begin{enumerate}
		\item (Existence and uniqueness of a ground state) For
		$c > c_0$, up to phase and translation, there exists a unique
		positive radial ground state $\ua \in H^{1}(\mathbb{R}^3)$
		which solves \eqref{1.2}.

		\item (Regularity) Let $\ua$ be the action ground state, then
		$\ua$ belongs to $H^s(\mathbb{R}^3)$ for each $s>0$.
		
		\item (Non-relativistic limit) Let $\ua$ ($\uai$) be the
		unique positive radial action ground state of the variational
		problem \eqref{1.2} (resp., \eqref{1.3}), then 
		\begin{equation*}
			\ua \to \uai \quad \text{in} \quad
			H^1(\mathbb{R}^3),\quad \text{as} \quad c\to \infty,\\
		\end{equation*}
    \end{enumerate}
\end{theorem}

\begin{remark}
	For the existence and regularity of a positive, radial action ground
	states see for example \cite{MR2799908}.
	
	About the uniqueness of the action ground state for $c$ large and
	the non-relativistic limit of the action ground state, we have not
	seen an accurate and detailed proof in the relevant literature.
	For this reason we give a proof in section \ref{sec:action}, see
	in particular Proposition \ref{Prop:A1}. Our method is related to
	the one used in \cite{MR3457922} to study the non-relativistic
	limit with a power nonlinearity. On this see also
	\cite{MR3734982} (paper in which the authors study the asymptotic
	behaviour for the pseudo-relativistic equation with the Hartree
	nonlinearity).
\end{remark}

To introduce our main results, let us observe that the symbol of the
operator 
\begin{equation*}
	P_{c}(D) = \sqrt{-c^2\Delta +m^2c^4}-mc^2
\end{equation*}
has the following Taylor series expansion for $\abs{\xi} < mc$
\begin{equation*}
	P_{c}(\xi) = \sqrt{c^2|\xi|^2 +m^2c^4}-mc^2 = P_\infty(\xi) +
	\sum_{j=1}^{\infty} \frac{P_{\infty,j}(\xi)}{c^{2j}},
\end{equation*}
where $P_\infty(\xi)= \frac{|\xi|^2}{2m}$,
$P_{\infty,j}(\xi)=\frac{(-1)^j\alpha_{j+1}}{m^{2j+1}}|\xi|^{2j+2}$,
$\alpha_j =\frac{(2j-2)!}{j!(j-1)!2^{2j-1}}$. This expansion leads to
the following asymptotic estimate for the pseudo-relativistic operator
$P_c(D)$ in the operator norm:
\begin{equation*}
	\left\| P_c(D) - P_\infty(D) - \sum_{j=1}^{n-1}
	\frac{P_{\infty,j}(D)}{c^{2j}} \right\|_{\mathcal{L}(H^{2n+2+s};
	H^s)} \lesssim \frac{1}{c^{2n}},
\end{equation*}
see Remark \ref{remark2.4} for details. In view of non-relativistic
limit, it is natural to ask whether the action or energy ground state
$u$ has the same asymptotic expansion, that is, for each $n\in
\mathbb{N}$, whether there exists $f_{1}, f_{2},\cdots, f_{n-1} $ such
that for any $s>0$,
\begin{equation}\label{1.12}
	\left\|u -u_\infty- \sum_{j=1}^{n-1}\frac{f_{j}}{c^{2j}}
	\right\|_{H^s}\lesssim\frac{1}{c^{2n}}.
\end{equation}
When $n=1$, \eqref{1.12} has been proved in \cite{MR3734982} for
action ground state, where the non-degeneracy of the non-relativistic
ground state plays a crucial role. 

In our paper, we consider the asymptotic properties of this solution
like \eqref{1.12} for each $n\in \mathbb{N}$ as $c\to \infty$.
Compared with \cite{MR3734982}, we adopt a different approach to
obtain the regularity of ground state in higher order Sobolev space,
and our results extend that of \cite{MR3734982} to each order
convergence rate. We also give details on the non-relativistic limit
of the action ground states.

\begin{theorem}[Asymptotic properties of action ground states]
	\label{them:1.3}
	Let $c > c_{0}$ and $\ua$, $\uai$ be the unique positive radial
	action ground state for the equation \eqref{1.2} and \eqref{1.3}
	respectively. Then for each $j \in \mathbb{N}$ there exists a
	unique radial function $\fa_{ j} \in \bigcap\limits_{s>0}
	H^{s}(\mathbb{R}^3)$, such that for all $n\in\mathbb{N}$
    \begin{equation*}
		\left\|\ua - \uai - \sum_{j=1}^{n-1} \frac{\fa_{j}}{c^{2j}}
		\right\|_{H^s} \lesssim \frac{1}{c^{2n}}.
    \end{equation*}
	The functions $\fa_{j}$ $(j\geq 1)$ are defined recursively as
	solutions of the following linear differential equation
	\begin{equation}
		\La \fa_j = - \sum_{k=0}^{j-1} P_{\infty, j-k}(D) \fa_{ k} +
		\sum_{k=2}^{\min\{j,3\}}\frac{1}{k!}
		\sum_{\substack{i_1+i_2+\cdots+i_k = j\\1\leq i_1,
		i_2,\cdots, i_k \leq j-1}}
		\mathcal{N}^{(k)}(\fa_0)[\fa_{i_1}, \fa_{i_2},\cdots,
		\fa_{i_k}],
    \end{equation}
	where $\La$ is the linearization of the nonlinear Hartree equation
	\eqref{1.3} around the solution $\uai$ (see Lemma \ref{Lem:3.1}),
	$\mathcal{N}^{(k)}$ is the $k$-th derivative of $\mathcal{N}$,
	$f_0 = u_\infty$.
\end{theorem}

We have a result also for the energy ground states. This result is new
and its proof requires estimate on the asymptotic behavior of the
Lagrange multiplier $\omega_{c}$ (see Lemma \ref{newlemma:1} and Lemma
\ref{newlemma:2}).
\begin{theorem}[Asymptotic properties of energy ground states]
	\label{them:1.5}
	Let $c > c_{0}$ and $\ue$, $\uei$ be the unique positive radial
	energy ground state of the variational problem \eqref{1.8} and
	\eqref{1.9} respectively. For each $j \in \mathbb{N}$, there
	exists a unique radial function $\fe_{j} \in \bigcap \limits_{s>0}
	H^{s}(\mathbb{R}^3)$ and constants $a_j$, $b_j\in \mathbb{R}$,
	such that for all $n \in \mathbb{N}$
	\begin{equation*}
		\left\| \ue - \uei - \sum_{j=1}^{n-1} \frac{\fe_{j}}{c^{2j}}
		\right\|_{H^s} \lesssim \frac{1}{c^{2n}},
	\end{equation*}	
	and 
	\begin{align*}
		&\left|e_{c} - e_{\infty} - \sum_{j=1}^{n-1}
		\frac{a_{j}}{c^{2j}} \right|\lesssim\frac{1}{c^{2n}}, 
		&&\left|\omega_{c} - \omega_{\infty} - \sum_{j=1}^{n-1}
		\frac{b_{j}}{c^{2j}}\right| \lesssim \frac{1}{c^{2n}}.
	\end{align*}	
	The functions $\fe_{j}$ $(j \geq 1)$ are defined recursively
	as solutions of the following linear differential equation
	\begin{equation}
		\begin{split}
			\Le \fe_j =& - \sum_{k=0}^{j-1} P_{\infty, j-k}(D) \fe_{
			k} -\sum_{k=0}^{j-1} b_{j-k}\fe_k\\
			&+ \sum_{k=2}^{\min\{j,3\}} \frac{1}{k!}
			\sum_{\substack{i_1+i_2+\cdots+i_k = j\\1\leq i_1,
			i_2, \cdots, i_k \leq j-1}}
			\mathcal{N}^{(k)}(\fe_0)[\fe_{i_1}, \fe_{i_2},\cdots,
			\fe_{i_k}],
		\end{split}
    \end{equation}
	where $\Le$ is the linearization of the nonlinear Hartree equation
	\eqref{1.11} around the solution $\uei$ (see Lemma \ref{Lem:3.1}),
	$\fe_0 = \uei$, $a_{1} = -\frac{1}{8m^3} \|\Delta \uei
	\|_{L^2}^2$, $b_1 = \frac{5}{8m^3} \|\Delta \uei \|_{L^2}^2$. The
	constants and $a_{j}$ and $b_{j}$ are defined in terms of the
	$\fe_{i}$ for $i = 1, \ldots j-1$, see the Lemmas \ref{newlemma:1}
	and \ref{newlemma:2} for their construction.
\end{theorem}

\begin{remark}
	Our results deals only with the nonlinearity $\mathcal{N}(\psi)$ of
	Hartree-type. It could be interesting to see if they can be
	extended to cover the case of other nonlinearities (for example
	the cubic one) or the case of boosted ground states, that is time
	dependent solutions of the form $\psi(t,x) = e^{i\lambda t} u(x -
	\nu t)$ for some $\nu \in \mathbb{R}^{3}$, $\abs{\nu} \leq 1$.
	
	On this last problem see \cite{MR2318846} and the recent results
	by Chen and Wang in \cite{zbMATH08142333} where the
	non-relativistic limit for such solutions is studied with
	techniques similar to the one developed in \cite{MR3734982}.
\end{remark}

\subsection*{Notations}
Throughout this paper, we make use of the following notations.
\begin{itemize}
	\item $\|\cdot\|_{L^q}$ denotes the usual norm of the Lebesgue
	space $L^q\left(\mathbb{R}^3\right)$; 
	
	\item $\|\cdot\|_{H^s}$ denotes the usual norm of the Sobolev
	space $H^s\left(\mathbb{R}^3\right)$; 
	
	\item $\scalar{\cdot}{\cdot}$ denotes the usual scalar product of
	the Hilbert space $L^2\left(\mathbb{R}^3\right)$; 
	
	\item If ${F} \colon H \to \mathbb{R}$ with $d{F}(u)$ we denote
	the Frechet derivative of ${F}$;

	\item $C$ is some positive constant that may change from line to
	line; \item $a\lesssim b$ means that $a \leq C b$;

	\item $\|P\|_{\mathcal{L}(H;K)}$ denotes the operator norm of the
	bounded linear operator $P$ from the Banach space $H$ to $K$.
\end{itemize}

\section{Notation and preliminary results}
We let
\begin{equation*}
	P_c(D):=\sqrt{-c^2\Delta +m^2c^4}-mc^2
\end{equation*}
and, for all $u \in H^{1/2}$,
\begin{equation*}
	\mathcal{H}(u) =\int_{\mathbb{R}^{3} \times \mathbb{R}^{3}}
	\frac{u^{2}(x) u^{2}(y)}{\abs{x - y}}.
\end{equation*}
Then, for all $u$, $v \in H^{1/2}(\mathbb{R}^{3})$
\begin{equation*}
	d\mathcal{H}(u)[v] = 4\scalar{\mathcal{N}(u)}{v}
\end{equation*}
and, letting
\begin{equation*}
	\scalar{P_{c}(D)u}{v} = \int_{\mathbb{R}^{3}}
	(\sqrt{c^2\abs{\xi}^{2} +m^2c^4} - mc^2) \mathcal{F}[u](\xi)
	\mathcal{F}[v](\xi) \, d\xi
\end{equation*}
we have that 
\begin{equation*}
	\mathcal{J}_{c}(u)= \scalar{P_{c}(D)u}{u} + \lambda
	\norm{u}^{2}_{L^{2}} - \frac{1}{2} \mathcal{H}(u)
\end{equation*}
and
\begin{equation*}
	d\mathcal{J}_{c}(u)[v] = 2\scalar{P_{c}(D)u}{v} + 2\lambda
	\scalar{u}{v} -  2\scalar{\mathcal{N}(u)}{v}.
\end{equation*}

The symbol of operator $P_c(D):=\sqrt{-c^2\Delta +m^2c^4}-mc^2$ has
Taylor expansion converging for $\abs{\xi} < m c$
\begin{equation*}
	P_c(\xi)= \sqrt{c^2|\xi|^2 +m^2c^4}-mc^2 = mc^2
	\left(\sqrt{\frac{|\xi|^2}{m^2c^2}+1} -1 \right) =
	\sum_{k=1}^{\infty} \frac{(-1)^{k-1}\alpha_k}{m^{2k-1}c^{2k-2}}
	|\xi|^{2k}
\end{equation*}
where $\alpha_k =\frac{(2k-2)!}{k!(k-1)!2^{2k-1}}$. We denote the
difference between $\sqrt{c^2|\xi|^2 +m^2c^4}-mc^2$ and first $n$
terms of its Taylor expansion by $P_{c,n}(\xi)$, that is
\begin{equation*}
    P_{c,n}(\xi):=  \sqrt{c^2|\xi|^2 +m^2c^4}-mc^2
    -\sum_{k=1}^{n}\frac{(-1)^{k-1}\alpha_k}{m^{2k-1}c^{2k-2}}|\xi|^{2k},
\end{equation*}
and the corresponding operator is
\begin{equation*}
	P_{c,n}(D) := \sqrt{-c^2\Delta +m^2c^4} - mc^2 - \sum_{k=1}^{n}
	\frac{(-1)^{k-1}\alpha_k}{m^{2k-1}c^{2k-2}}(-\Delta)^k.
\end{equation*}
For convenience, we set
\begin{equation*}
	P_{\infty,n}(D):=\frac{(-1)^{n}\alpha_{n+1}}{m^{2n+1}}(-\Delta)^{n+1},
	\,\,P_\infty(D):=P_{\infty, 0}(D)= -\frac{\alpha_1}{m}\Delta,
\end{equation*}
then
\begin{equation*}
    P_{c,n+1}(D) = P_{c,n}(D) - \frac{1}{c^{2n}}P_{\infty, n}(D).
\end{equation*}

\begin{lemma}\label{Lem:4.1}
	For all $c > 0$ and $n \in \mathbb{N}$ the operator $P_{c, n}(D)$
	is elliptic of order $2n$. Its symbol $P_{c,n} (\xi)$ satisfies
	\begin{equation*}
		\abs{P_{c, n}(\xi)} \geq C\abs{\xi}^{2n} \qquad \text{for $\xi$ 
		large,}
	\end{equation*}
	and the following holds for all $\xi \in \mathbb{R}^{3}$:
    \begin{align}
		P_{c, n}(\xi) 
		\begin{cases}
			\ \leq 0 & \text{ n is odd} \\
			\ \geq 0 & \text{ n is even}
		\end{cases}
    \end{align}
	and
	\begin{equation}
		\label{eq:PcnEstimate}
		\abs{P_{c, n}(\xi)} \leq
		\frac{C_{n}|\xi|^{2n+2}}{m^{2n+1}c^{2n}}.
	\end{equation}
	In particular
	\begin{equation*}
		\abs{P_{c, 2}(\xi)} = \labs{\sqrt{c^2|\xi|^2 +m^2c^4}- mc^2 -
		\frac{\abs{\xi}^{2}}{2m} + \frac{\abs{\xi}^{4}}{8m^{3}c^{2}}}
		\leq \frac{C_{2}|\xi|^{6}}{m^{5}c^{4}}.
	\end{equation*}
\end{lemma}

\begin{remark}
	Inequalities of the type $0 \leq P(\xi) \leq Q(\xi)$ between the 
	symbols of the differential operators $P(D)$ and $Q(D)$ will be 
	used to deduce that
	\begin{equation*}
		0 \leq \scalar{P(\xi) u}{u} \leq \scalar{Q(\xi) u}{u}
	\end{equation*}
	for all $u \in H^{s}$ (with $s$ sufficiently large). 
	
	We will use the inequalities of Lemma \ref{Lem:4.1} in the proofs
	of Lemma \ref{Lem:A2} and \ref{Lem:A3}.
\end{remark}

\begin{proof}
	Set
	\begin{equation*}
		f(t)= \sqrt{t + 1} - 1, \quad t \geq 0.
	\end{equation*}
    According to Taylor's formula
    \begin{align*}
		f(t) &= \sum_{j=0}^n \frac{f^{(j)}(0)}{j!} t^j +
		\frac{f^{(n+1)}(\tau)}{(n+1)!}t^{n+1}, \quad \tau\in [0, t] \\
		&= \sum_{j=1}^{n}{(-1)^{j-1}\alpha_j}t^{j} +
		{(-1)^{n} \alpha_{n+1}} t^{n+1}(\tau+1)^{-\frac{2n+1}{2}},
    \end{align*}
    and hence
	\begin{equation}
		\label{eq:taylor}
		\sqrt{\xi +1}-1-\sum_{j=1}^{n}{(-1)^{j-1}\alpha_j}\xi^{j}
		\quad
		\begin{cases}
			\leq 0  & \text{$n$ odd} \\
			\geq 0 & \text{$n$ even}
		\end{cases}.
	\end{equation}
    From this immediately we deduce
	\begin{equation}
		\label{eq:taylormc}
		P_{c, n}(\xi) = mc^2 \left( \sqrt{\frac{|\xi|^{2}}{m^{2}c^{2}}
		+ 1} - 1 - \sum_{j=1}^{n} {(-1)^{j-1}\alpha_j}
		\left(\frac{|\xi|^{2}}{m^{2}c^{2}} \right)^{j} \right) \
		\begin{cases}
			\ \leq 0  & \text{$n$  odd} \\
			\ \geq 0 & \text{$n$ even}
		\end{cases}.
	\end{equation}
	We also deduce that
	\begin{equation*}
		\left|P_{c,n}(\xi)\right| = mc^2\left| \sqrt{
		\frac{|\xi|^2}{m^2c^2} +1 }-1 -
		\sum_{j=1}^{n}(-1)^{j-1}\alpha_j \left( \frac{|\xi|^2}{m^2c^2}
		\right)^j \right|\leq
		\frac{C_{n}|\xi|^{2n+2}}{m^{2n+1}c^{2n}}.
    \end{equation*}
\end{proof}

\begin{lemma}\label{Lem:2.1}
	For $f\in H^{2n + 2 + s}(\mathbb{R }^3)$, $s\geq0$, there exists a
	constant $C_{n} >0$, independent of $c$ and $m$, such that
    \begin{equation*}
		\left\| P_{c,n}(D)f \right\|_{H^s}\leq
		\frac{C_{n}}{m^{2n+1}c^{2n}}\|f\|_{H^{2n+2+s}}.
    \end{equation*}
\end{lemma}
\begin{proof}
	From \eqref{eq:PcnEstimate} we have
    \begin{multline*}
		\left\| P_{c,n}(D)f \right\|_{H^s} = \left\| (1+
		|\xi|^2)^{s/2}P_{c,n}(\xi)\mathcal{F}[f] \right\|_{L^2} \\
		\leq \frac{C_{n,m}}{c^{2n}} \left\|(|\xi|^{2n+2+s}+1)
		\mathcal{F}[f]\right\|_{L^2} \leq \frac{C_{n}}{m^{2n+1}c^{2n}}
		\|f\|_{H^{2n+2+s}}.
    \end{multline*}
\end{proof}

\begin{remark}\label{remark2.4}
	Lemma \ref{Lem:2.1} implies that $P_{c,n}(D)$, as a bounded linear
	operator from $H^{2n+2+s}$ to $H^s$, has its norm dominated by
	$c^{-2n}$, i.e.,
	\begin{equation*}
		\left\|P_c(D)-P_\infty(D)-\sum_{j=1}^{n-1}
		\frac{P_{\infty,j}(D)}{c^{2j}}
		\right\|_{\mathcal{L}(H^{2n+2+s}; H^s)} \lesssim
		\frac{1}{c^{2n}},
	\end{equation*}
	and it is easy to see that ${c^{-2n}}$ is the optimal bound.
\end{remark}

The following trilinear estimate in \cite[Lemma 3.2]{MR3734982} is of
fundamental importance for proving the higher regularity of the
positive solutions to equation \eqref{1.2}.
\begin{lemma}\label{Lem:2.2}
	For $s\geq \frac{1}{2}$, $u_1, u_2, u_3\in H^s(\mathbb{R}^3)$,
	there holds 
	\begin{equation*}
		\left\|\left(|x|^{-1} *\left(u_1 u_2\right)\right)
		u_3\right\|_{H^s\left(\mathbb{R}^3\right)} \lesssim
		\prod_{j=1}^3\left\|u_j\right\|_{H^s\left(\mathbb{R}^3\right)}.
	\end{equation*}
\end{lemma}
For convenience, we denote $\mathcal{N}^{(n)}$ as the $n$-th
derivative of nonlinear mapping
\begin{equation*}
	\mathcal{N} : H^s(\mathbb{R}^3)\to
	H^s(\mathbb{R}^3).
\end{equation*}
Then, direct computation shows
\begin{gather*}
	\mathcal{N}^{(1)}(u)[h_1] = {\left( |x|^{-1} * u^2 \right) h_1 + 2
	\left( |x|^{-1} * (u h_1) \right) u}, \\
	\mathcal{N}^{(2)}(u)[h_1, h_2] ={2 \left( |x|^{-1} * (h_1 h_2)
	\right) u + 2 \left( |x|^{-1} * (u h_1) \right) h_2 + 2 \left(
	|x|^{-1} * (u h_2) \right) h_1}, \\
	\mathcal{N}^{(3)}(u)[h_1, h_2, h_3] ={2
	\sum_{\substack{\text{cyclic} \\ i,j,k}} \left( |x|^{-1} * (h_i
	h_j) \right) h_k},
\end{gather*}
and $\mathcal{N}^{(n)}(u)=0$ for $n\geq 4$.
\begin{lemma}[Taylor's formula]\label{Lem:taylor}
    For $s\geq \frac{1}{2}$, $u, h\in H^s(\mathbb{R}^3)$, there holds
    \begin{equation}
		\mathcal{N}(u+ h) =\sum_{k=0}^n\frac{1}{k!}
		\mathcal{N}^{(k)}(u)[\underbrace{h,\cdots,h}_{k\,\,
		\text{times}}]+ \frac{1}{n!}\int_{0}^{1}(1-t)^{n}
		\mathcal{N}^{(n+1)}(u+ th)[\underbrace{h,\cdots,h}_{n+1\,\,
		\text{times}}]dt.
    \end{equation}
\end{lemma}
By the Taylor's formula, the following nonlinear estimation can be
directly obtained.
\begin{lemma}
	\label{Lem:2.4}
    For $s\geq \frac{1}{2}$, $u, h\in H^s(\mathbb{R}^3)$, there holds
    \begin{equation}
		\left\| \mathcal{N}(u+ h) - \sum_{k=0}^n \frac{1}{k!}
		\mathcal{N}^{(k)}(u)[\underbrace{h,\cdots,h}_{k\,\,
		\text{times}}]\right\|_{H^s} \leq C\|h\|_{H^s}^{n+1},\quad
		n=1, 2,
    \end{equation}
    and 
    \begin{equation*}
		\mathcal{N}(u+ h) = \sum_{k=0}^3 \frac{1}{k!}
		\mathcal{N}^{(k)}(u)[\underbrace{h,\cdots,h}_{k\,\,
		\text{times}}].
    \end{equation*}
	where $C$ is a constant which is dependent on $s$, $n$, and
	$\max\limits_{t\in [0,1]}\|u+th\|_{H^s}$.
\end{lemma}

Similarly, for the functional $\mathcal{E}_\infty$, we have the
following expansion.
\begin{lemma}
	\label{Lem:2.4_new}
	For $s \geq \frac{1}{2}$, $u$, $h\in H^s(\mathbb{R}^3)$, there
	holds
    \begin{equation*}
		\mathcal{E}_\infty(u+ h) = \sum_{k=0}^4 \frac{1}{k!}
		d^k\mathcal{E}_\infty(u)[\underbrace{h,\cdots,h}_{k\,\,
		\text{times}}].
    \end{equation*}
\end{lemma}

We recall the following convergence result
\begin{lemma}[{\cite[Lemma 4.1]{MR3148610}}]\label{lemm:2.5}
	Let $v\in H^{1/2}(\mathbb{R}^3)$, $f_n, g_n, h_n$ bounded
	sequences in $H^{1/2}(\mathbb{R}^3)$, and one of them converges
	weakly to zero in $H^{1/2}(\mathbb{R}^3)$, then
    \begin{equation*}
		\int_{\mathbb{R}^3\times\mathbb{R}^3}
		\frac{|f_n|(x)|g_n|(x)|v|(y)|h_n|(y)}{|x-y|} \, dxdy \to 0,
		\quad \text{as} \quad n\to+\infty.
    \end{equation*}
\end{lemma}

\begin{lemma}\label{lemm:2.6_new}
   For $s\geq 0$, let $v\in H^{s+2}(\mathbb{R}^3)$, $h_n, g_n$ bounded
   sequences in $H^{s+2}(\mathbb{R}^3)$, and one of them converges
   weakly to zero in $H^{s+2}(\mathbb{R}^3)$, then
   \begin{equation*}
		\left\|(|x|^{-1}*vh_n)g_n \right\|_{H^s}\to
		0,\quad\mathrm{as}\quad n\to+\infty.
	\end{equation*}
\end{lemma}
\begin{proof}
	According to Lemma \ref{Lem:2.2}, the sequence $f_n=
	(|x|^{-1}*vh_n)g_n$ is bounded in $H^{s + 2}(\mathbb{R}^3)$. Then,
	by Lemma \ref{lemm:2.5}, we have
    \begin{equation}\label{new_1}
			\left\|f_n \right\|_{L^2}^2 \leq
			\int_{\mathbb{R}^3\times\mathbb{R}^3}
			\frac{|f_n(x)|\cdot|g_n(x)|\cdot|v(y)|\cdot|h_n(y)|}{|x -
			y|} \, dxdy \to 0 \quad \text{as} \quad n\to+\infty.
    \end{equation}
	There exists an integer $m$ such that $s\leq m < m + 1 \leq s +
	2$. If $m$ is even, then by Leibniz's formula for weak
	derivatives, $(-\Delta)^{\frac{m}{2}}f_n$ can be expressed as a
	linear combination of terms of the form
	$\left((|x|^{-1}*(\partial^\alpha v\partial^\beta
	h_n)\right)\partial^\gamma g_n$, where the multi - indices
	$\alpha, \beta, \gamma$ satisfy $|\alpha|+ |\beta| + |\gamma|=m$.
	By the given assumption, either $\partial^\beta h_n$ or
	$\partial^\gamma g_n$ converges weakly to zero in
	$H^{\frac{1}{2}}(\mathbb{R}^3)$. Then, applying Lemma
	\ref{lemm:2.5}, we obtain
    \begin{multline}\label{new_2}
        \|(-\Delta)^{\frac{m}{2}}f_n\|_{L^2}^2 \\
		\leq \sum_{|\alpha|+ |\beta|+
		|\gamma|=m}\int_{\mathbb{R}^3\times\mathbb{R}^3}
		\frac{|\Delta^{\frac{m}{2}}f_n(x)|\cdot|\partial^\gamma
		g_n(x)|\cdot|\partial^\alpha v(y)|\cdot|\partial^\beta
		h_n(y)|}{|x - y|} \, dxdy \to 0.
    \end{multline}
    If $m$ is odd, repeating the above argument, we get
    \begin{multline}\label{new_3}
		\|(-\Delta)^{\frac{m}{2}}f_n\|_{L^2}^2 =
		\scalar{(-\Delta)^{\frac{m - 1}{2}}f_n}
		{(-\Delta)^{\frac{m+1}{2}}f_n} \\
		\leq \sum_{|\alpha|+ |\beta|+ |\gamma|=m -
		1}\int_{\mathbb{R}^3\times\mathbb{R}^3}
		\frac{|(-\Delta)^{\frac{m+1}{2}}f_n(x)|\cdot|\partial^\gamma
		g_n(x)|\cdot|\partial^\alpha v(y)|\cdot|\partial^\beta
		h_n(y)|}{|x - y|} \, dxdy \to 0.
    \end{multline}
	By combining \eqref{new_1}, \eqref{new_2} and \eqref{new_3}, we
	conclude that
	\begin{equation*}
		\|f_n\|_{H^s} \leq \|f_n\|_{H^m}\to0 \quad\text{as}\quad
		n\to+\infty.
	\end{equation*}
\end{proof}

We recall the following Hardy-Littlewood-Sobolev inequality, see
\cite[Theorem 4.3]{liebloss}.

\begin{lemma}\label{hls}
	Assume that $f \in L^p(\mathbb{R}^3)$ and $g \in
	L^q(\mathbb{R}^3)$. Then one has
	\begin{equation*}
		\int_{\mathbb{R}^3} \int_{\mathbb{R}^3}
		\frac{f(x)g(y)}{|x-y|^t} \, dx \, dy \leq c(p, q, t) \|f\|_p
		\|g\|_q,
	\end{equation*}
	where $1 < p, q < \infty$, $0 < t < 3$ and $\frac{1}{p} +
	\frac{1}{q} + \frac{t}{3} = 2$.
 \end{lemma}

We will use some of the properties of the non-relativistic action and
energy ground states in proving the asymptotic behavior of the action
and energy ground states. In particular we recall that the unique
positive radial action ground state $\uai$ to \eqref{1.3} is
nondegenerate in $H^{1}_r(\mathbb{R}^3)$ (see \cite[Proposition
2]{MR2561169} and \cite[Theorem III.1]{zbMATH05643114}), that means
the following linearized equation has only the trivial solution $u=0$
in $H^{1}_r(\mathbb{R}^3)$,
\begin{align*}
	\La u: & = \left(P_\infty(D)+ \lambda \right)u-
	\mathcal{N}^{(1)}(\uai)u \\
	& = -\frac{1}{2m}\Delta u +\lambda u-
	\left(|x|^{-1}*\uai^2\right)u-2\left(|x|^{-1}*(\uai
	u)\right)\uai.
\end{align*}
Let us recall that existence and uniqueness of the positive radial
solution goes back to \cite{zbMATH03576003}, see also
\cite{zbMATH01340074}. The precise knowledge of $\ker \La \bigcap
H^{1}_r(\mathbb{R}^3)$ implies $\La \colon H_r^{s+2}(\mathbb{R}^3)\to
H_r^s(\mathbb{R}^3)$ is invertible. To prove it we start by showing:
\begin{lemma}
	\label{lemm:3.1_new}
	For all $s \geq 0$ the linear operator $\mathcal{N}^{(1)}(\uai)
	\colon H_{r}^{s+2} \to H_{r}^s$ is compact.
\end{lemma}

\begin{proof}
	Assume that the sequence $\{u_n\}$ converges weakly to zero in
	$H^{s + 2}(\mathbb{R}^3)$. Then, by Lemma \ref{lemm:2.6_new}, it
	follows that $\{\mathcal{N}^{(1)}(\uai)u_n \}$ converges strongly
	to zero in $H^{s}(\mathbb{R}^3)$.
\end{proof}

\begin{lemma}\label{Lem:3.1}
	For each $\lambda > 0$ and $s\geq 0$, the linearized operator
	$\La : H_r^{s+2}(\mathbb{R}^3)\to H_r^s(\mathbb{R}^3) $ is
	invertible.
\end{lemma}

\begin{proof}
    We have that
	\begin{equation*}
		\La = P_\infty(D) + \lambda
		-\mathcal{N}^{(1)}(\uai)= \left(Id -
		\mathcal{N}^{(1)}(\uai)(P_\infty(D) + \lambda)^{-1}
		\right)(P_\infty(D)+\lambda).
	\end{equation*}
	It can easily be seen that $P_\infty(D)+\lambda:
	H_r^{s+2}(\mathbb{R}^3)\to H_r^s(\mathbb{R}^3)$ is invertible, we
	need only prove the invertibility of $Id -
	\mathcal{N}^{(1)}(\uai)(P_\infty(D)+\lambda)^{-1}$. It
	follows from the non-degeneracy of linearized operator
	$\La$ that $\ker \La \bigcap H_r^1=0.$ By Lemma
	\ref{lemm:3.1_new}, the operator
	$\mathcal{N}^{(1)}(\uai)(P_\infty(D)+\lambda)^{-1}$ is a
	compact operator on $H_r^s(\mathbb{R}^3)$, and by Fredholm
	alternative, we obtain the invertibility of $Id -
	\mathcal{N}^{(1)}(\uai)(P_\infty(D) + \lambda)^{-1}$, and
	the result follows.
\end{proof}
We state here also the following lemma which follows from the 
invertibility of $\La$
\begin{lemma}\label{Lem:3.3}
	Let $ f_n, h_n$ be two sequences in $\bigcap\limits_{s>0}
	H_r^{s}(\mathbb{R}^3)$ which satisfy
    \begin{equation}
        h_n = \La f_n ,
    \end{equation}
    and for each $s>0$, there holds
	\begin{equation*}
		\|h_n\|_{H^s} = o_{n}(1) + o_{n}(1) \|f_n\|_{H^s}, \quad
		\text{as } n \to \infty,
	\end{equation*}
	then $\|f_n\|_{H^s} = o_{n}(1)$ as $n\to \infty.$
\end{lemma}

\begin{proof}
    Since $\La$ is invertible, there holds
	\begin{equation*}
		\|f_n\|_{H^{s+2}}=\|\La^{-1}h_n\|_{H^{s+2}} \lesssim
		\|h_n\|_{H^s} = o_{n}(1) + o_{n}(1)\|f_n\|_{H^s}.
	\end{equation*}
    Hence, for each $s>0$,
    \begin{equation*}
		\|f_{n}\|_{H^s} = o_{n}(1),
    \end{equation*}
	as $n\to \infty.$
\end{proof}

\section{Non-relativistic limit of the action ground state}
\label{sec:action}

In this section we fill in the proof of Theorem \ref{them:actionGS}
showing that the action ground state converges as $c \to +\infty$ to
the positive and radial action ground state for the equation
\eqref{1.3} $\uai$, which by Lemma 9 in \cite{MR2561169} is unique.
The uniqueness of $\ua$ for $c$ large will follow.

\begin{proposition}
	\label{Prop:A1}
	There exists $c_0 > 0$, such that for $c > c_0$, equation
	\eqref{1.2} has a unique radial and positive action ground state
	$\ua$. Moreover, for each $s > 0$, there holds
	\begin{equation*}
		\ua \to \uai \quad \text{in} \quad H^s(\mathbb{R}^3),\quad
		\text{as} \quad c\to \infty.
	\end{equation*}
\end{proposition}
Regarding the existence of $\ua$, this is a standard variational
problem, and we omit the details, see for example \cite{MR2799908},
although there the ground state is found via the Mountain Pass
theorem.

Our main focus is on proving its non-relativistic limit. Let's recall
that action ground state $\ua$ ($\uai$) for the equation \eqref{1.2}
(resp., \eqref{1.3}) satisfies
\begin{equation*}
	\mathcal{J}_c(\ua)= \inf_{u \in \mathcal{M}_c} \mathcal{J}_c(u),
	\qquad \mathcal{J}_\infty(\uai) = \inf_{u \in
	\mathcal{M}_\infty} \mathcal{J}_\infty(u)
\end{equation*}
where 
\begin{equation*}
	\mathcal{M}_c= \settc{u\in H^{1/2}(\mathbb{R}^3)\setminus \{0\}}
	{d\mathcal{J}_c(u)[u] = 0},
\end{equation*}
and 
\begin{equation*} 
	\mathcal{M}_\infty= \settc{u\in H^{1}(\mathbb{R}^3)\setminus
	\{0\}}{ d\mathcal{J}_\infty(u)[u] = 0 }.
\end{equation*}

\begin{lemma}\label{Lem:A2}
	For each $c>1$, there exists $t_c>0$, such that $t_c \uai \in
	\mathcal{M}_c$, and there holds
    \begin{equation*}
		\lim_{c \to \infty} t_c = 1.
    \end{equation*}
\end{lemma}

\begin{proof}
	Since $\uai  \in \neh_{\infty}$ we have that
	\begin{equation*}
		\scalar{(P_{\infty}(D) + \lambda) \uai }{\uai } = 
		\mathcal{H}(\uai ).
	\end{equation*}
	For all $u \in H^{1/2}(\mathbb{R}^{3})$
	\begin{equation*}
		d\mathcal{J}_c(u)[u] = 2\scalar{(P_{c}(D) + \lambda) u}{u} -
		2 \mathcal{H}(u)
	\end{equation*}
	and hence $t_c \uai \in \mathcal{M}_c$ provided $t_c$
	is such that
    \begin{align*}
		0 &= 2\scalar{(P_{c}(D) + \lambda) \uai }{\uai } =
		2t_c^2 \mathcal{H}(\uai )\\
		&=2t_c^2 \scalar{(P_{\infty}(D) +\lambda)
		\uai }{\uai }.
    \end{align*}
    Then
    \begin{equation*}
    	(t^{2}_{c} - 1) \scalar{(P_{\infty}(D) +\lambda)
		\uai }{\uai } ) = \scalar{(P_{c}(D) - 
		P_{\infty}(D))\uai }{\uai },
    \end{equation*}
    It follows from Lemma \ref{Lem:4.1} that
    \begin{equation}\label{eq_A}
		0 \leq P_{\infty}(\xi) - P_{c}(\xi) = \frac{|\xi|^2}{2m} -
		\left(\sqrt{c^2|\xi|^2 +m^2c^4}-mc^2\right) \leq
		\frac{|\xi|^4}{8m^3c^2}.
    \end{equation}
    and
    hence
	\begin{equation*}
		|t_c^2-1| \scalar{(P_{\infty}(D) +\lambda) \uai }{\uai } )
		\leq \frac{1}{8m^3c^2} \int_{\mathbb{R}^3}
		\mathcal{F}[\uai]^2(\xi) |\xi|^4 d\xi,
	\end{equation*}
    which yields $\lim\limits_{c\to \infty} t_c=1.$
\end{proof}

\begin{lemma}\label{Lem:A3}
	$\limsup\limits_{c\to \infty}\mathcal{J}_c(\ua)\leq
	\mathcal{J}_\infty(\uai)$.
\end{lemma}

\begin{proof}
	From Lemma \ref{Lem:4.1} we have that
	\begin{equation*}
		\sqrt{-c^2\Delta+m^2c^4} - mc^2 \leq -\frac{\Delta}{2m}.
	\end{equation*}
	It then follows Lemma \ref{Lem:A2} that
	\begin{equation*}
		\limsup_{c\to \infty} \mathcal{J}_c(\ua)  \leq
		\limsup_{c\to \infty}\mathcal{J}_c(t_c \uai ) \leq
		\limsup_{c\to \infty} \mathcal{J}_\infty(t_c \uai)
		=\mathcal{J}_\infty(\uai).
    \end{equation*}
\end{proof}

\begin{lemma}\label{Lem:A4}
    $\sup\limits_{c>1}\|\ua\|_{H^{1/2}}<\infty.$
\end{lemma}
\begin{proof}
	Given that $\ua$ satisfies
	\begin{equation*}
		\scalar{(P_{c}(D) + \lambda)\ua}{\ua} = \mathcal{H}(\ua)
	\end{equation*}
	then $\mathcal{J}_{c}(\ua) = \frac{1}{2} \scalar{(P_{c}(D) + 
	\lambda)\ua}{\ua}$ and, by Lemma \ref{Lem:A3}, we get
	\begin{equation*}
		\sup_{c>1} \scalar{(P_{c}(D) + \lambda)\ua}{\ua} =
		\sup_{c>1} 2\mathcal{J}_c(\ua) < \infty.
	\end{equation*}
	There exists a constant $C > 0$ which depends on $\lambda > 0$
	and $m$, such that for $c>1$,
	\begin{equation}
		\label{eq:coercive}
		P_{c}(\xi)+ \lambda = \sqrt{c^2\abs{\xi}^{2} + m^2c^4} - mc^2
		+ \lambda \geq C \sqrt{\abs{\xi}^{2} + 1}.
	\end{equation}
	Therefore, we have
	\begin{equation*}
		\sup_{c>1} \|\ua\|_{H^{1/2}}^2 \lesssim
		\sup_{c>1} \scalar{(P_{c}(D) +
		\lambda)\ua}{\ua}<\infty.
	\end{equation*}
\end{proof}
As a corollary to Lemma \ref{Lem:A4} and by employing the method of
\cite[Lemma 3]{MR2561169}, we can obtain the following lemma.
\begin{lemma}
	$\sup\limits_{c>1}\|\ua\|_{H^{1}}<\infty.$
\end{lemma}
In  \cite[Proposition 4.2]{MR3734982}, the authors show that the ground
state to \eqref{1.2} are uniformly bounded in higher order Sobolev
spaces. In their proof, the operator inequality \eqref{eq:coercive}
played a crucial role. Here, we adopt a different approach to prove
that $\ua$ are uniformly bounded in $H^s(\mathbb{R}^3)$.

\begin{lemma}\label{Lem:2.6}
	Let $\ua $ be the action ground state to \eqref{1.2}, then
	for each $s > 0$
	\begin{equation*}
		\sup_{c> 1}\|\ua\|_{H^s}<\infty.
	\end{equation*}
\end{lemma}
\begin{proof}
	Acting with the operator $(-\Delta)^s\sqrt{-c^2\Delta + m^2c^4}$
	on both sides of \eqref{1.2}, we have
    \begin{equation}
        \begin{split}
			c^2(-\Delta)^{s+1} \ua & = \sqrt{-c^2\Delta
			+m^2c^4}(-\Delta)^s( (|x|^{-1}* \ua ^2)\ua ) \\
			& +(mc^2-\lambda )(-\Delta)^s( (|x|^{-1}* \ua^2)\ua )
			+(-\Delta)^s(\lambda ^2-2mc^2\lambda )\ua,
        \end{split}
    \end{equation}
    then,
    \begin{equation}\label{2.4}
		\begin{split}
			c^2\|(-\Delta)^{s+1} \ua \|_{L^2} & \lesssim \|
			\sqrt{-c^2\Delta +m^2c^4} (-\Delta)^s( (|x|^{-1}*
			\ua^2)\ua )\|_{L^2} \\
			& \quad+ c^2\|(-\Delta)^s( (|x|^{-1}* \ua^2)\ua  )\|_{L^2}
			+c^2\|(-\Delta)^s \ua \|_{L^2} \\
			& =: K_1 + K_2,
		\end{split}
    \end{equation}
    where
    \begin{align*}
		&K_1 = \|\sqrt{-c^2\Delta +m^2c^4}(-\Delta)^s( (|x|^{-1}*
		\ua^2)\ua )\|_{L^2} ,\\
        &K_2 = c^2\|(-\Delta)^s( (|x|^{-1}* \ua^2)\ua )\|_{L^2}
        +c^2\|(-\Delta)^s\ua\|_{L^2}.
    \end{align*}
    If $0\leq s< \frac{1}{4}$, we have operator inequality
    \begin{equation}\label{2.5}
        (-\Delta)^s\leq 4s(-\Delta)^{\frac{1}{4}} + 1-4s,
    \end{equation}
	which follows from Young inequality $ab\leq \frac{a^p}{p} +
	\frac{b^q}{q}$ with $\frac{1}{p}+\frac{1}{q}=1.$ Then, Lemma
	\ref{Lem:2.2} implies,
    \begin{equation}\label{2.6}
		K_2\lesssim c^2\left(\||x|^{-1}* \ua^2)\ua \|_{H^{1/2}}^3 +
		\|\ua\|_{H^{1/2}} \right) \lesssim
		c^2\left(\|\ua\|_{H^{1/2}}^3 + \|\ua\|_{H^{1/2}} \right).
    \end{equation}
    If $s\geq \frac{1}{4}$, by Lemma \ref{Lem:2.2} again
    \begin{equation}\label{2.7}
		K_2 \lesssim c^2\left(\|\ua\|_{H^{2s}}^3 +
		\|\ua\|_{H^{2s}}\right).
    \end{equation}
    To estimate $K_1$, from the operator inequality
    \begin{equation}\label{2.8}
        \sqrt{-c^2\Delta+m^2c^4}\leq c(-\Delta)^{\frac{1}{2}}+ mc^2,
    \end{equation}
    and combine \eqref{2.8} with \eqref{2.5}, we are led to
    \begin{equation}
		\label{2.9}
		K_1 \lesssim 
		\begin{cases}
			c \|\ua\|_{H^{2s+1}}^3 + c^2 \|\ua\|_{H^{1/2}}^3,
			& 0 \leq s < \frac{1}{4},\\
			c\|\ua\|_{H^{2s+1}}^3 + c^2 \|\ua\|_{H^{2s}}^3,
			&\frac{1}{4} <s.
		\end{cases}
    \end{equation}
	Consequently, by \eqref{2.4}, \eqref{2.6}, \eqref{2.7} and
	\eqref{2.9}, we get
	\begin{equation}
		\label{2.10}
		\|\ua\|_{H^{2s+2}} \lesssim 
		\begin{cases}
			\frac{1}{c} \|\ua\|_{H^{2s+1}}^3 + \|\ua\|_{H^{1/2}}^3
			+ \|\ua\|_{H^{1/2}} & 0 \leq s < \frac{1}{4}, \\
			\frac{1}{c} \|\ua\|_{H^{2s+1}}^3 + \|\ua\|_{H^{2s}}^3 +
			\|\ua\|_{H^{2s}} & \frac{1}{4} < s.
		\end{cases}
	\end{equation}
    By combining \eqref{2.10}
    with the fact that
	\begin{equation*}
		\sup_{c>1}\|\ua\|_{H^1}  < \infty,
	\end{equation*}
	we conclude that for each
	$s>0$
	\begin{equation*}
		\sup_{c>1}\|\ua\|_{H^s}  < \infty.
	\end{equation*}
\end{proof}

\begin{lemma}\label{Lem:A7}
	$\lim\limits_{c\to \infty}\mathcal{J}_\infty(\ua) =
	\mathcal{J}_\infty(\uai)$.
\end{lemma}
\begin{proof}
	Repeating the argument of Lemma \ref{Lem:A2}, for each $c>1$,
	there exists $r_c>0$, such that $r_c \ua \in \mathcal{M}_\infty$,
	and there holds
	\begin{equation*}
		\lim_{c\to \infty} r_c=1.
	\end{equation*}
    Then, 
	\begin{equation*}
		\mathcal{J}_\infty(\uai) \leq \mathcal{J}_\infty(r_c \ua) \leq
		\liminf\limits_{c\to \infty}\mathcal{J}_\infty(\ua).
	\end{equation*}
    On the other hand, by operator inequality (see Lemma 
	\ref{Lem:4.1}) we have
	\begin{equation*}
		- \frac{\Delta}{2m}\leq \sqrt{-c^2\Delta + m^2c^4}-mc^2 +
		\frac{\Delta^2}{8m^3c^2},
	\end{equation*}
	and hence we have
	\begin{equation*}
		\mathcal{J}_\infty(\ua) \leq \mathcal{J}_c(\ua) +
		\frac{1}{8m^3c^2}\|\Delta \ua\|_{L^2}^2.
	\end{equation*}
	By combining this with Lemma \ref{Lem:A3} and Lemma \ref{Lem:2.6},
	there holds
	\begin{equation*}
		\limsup\limits_{c\to \infty}\mathcal{J}_\infty(\ua)\leq
		\mathcal{J}_\infty(\uai),
	\end{equation*}
	which completes the proof of the lemma.
\end{proof}

\begin{lemma}\label{Lem:A8}
	$\{\ua\}$ is a bounded Palais-Smale sequence for the functional
	$\mathcal{J}_\infty(u)$ at level $\mathcal{J}_\infty(\uai).$
\end{lemma}

\begin{proof}
    It follows from \eqref{eq_A} and Lemma \ref{Lem:2.6} that
    \begin{align*}
		\sup_{\|h\|_{H^1} \leq 1} |d\mathcal{J}_\infty(\ua)[h]| &=
		\sup_{\|h\|_{H^1} \leq 1} |d\mathcal{J}_\infty(\ua)[h] -
		d\mathcal{J}_c(\ua)[h]| \\
		&= 2\sup_{\|h\|_{H^1} \leq 1} \scalar{(P_{c}(D) - 
		P_{\infty}(D))\ua}{h} \\
		&\leq \frac{1}{4m^{3}c^{2}} \sup_{\|h\|_{H^1} \leq 1}
		\int_{\mathbb{R}^3} |\xi|^4|
		\mathcal{F}[{u}_c]|(\xi)|\mathcal{F}[{h}]|(\xi) \, d\xi \leq
		\frac{C}{c^2},
    \end{align*}
	and the lemma follows taking into account Lemma \ref{Lem:A7}.
\end{proof}
\begin{lemma}
    $\|\ua - \uai\|_{H^1}\to 0$ as $c\to \infty$.
\end{lemma}
\begin{proof}
	Since $\{\ua\}$ is radial and bounded sequence in
	$H^1(\mathbb{R}^3)$, by compact embedding
	$H^1_r(\mathbb{R}^3)\hookrightarrow L^p$ $(2<p<6)$, up to a
	sequence, we can assume that $\ua$ converge weakly to $u$ in
	$H^1(\mathbb{R}^3)$ and $\|\ua - u\|_{L^p} \to 0$. Then, $u \geq
	0$ is a weak solution to \eqref{1.3}. By Lemma \ref{hls} and
	Hölder inequality,
    \begin{align*}
		\int_{\mathbb{R}^3 \times \mathbb{R}^3}
		\frac{\ua^2(x) \ua(y)(\ua(y) - u(y))}{|x-y|} \, dx dy \leq &C
		\|\ua\|_{L^{12/5}}^2 \|\ua(\ua - u)\|_{L^{6/5}}\\
		\leq &C\|\ua\|_{L^{12/5}}^2\|\ua\|_{L^2}\|\ua - u\|_{L^3}\\
		\to& 0, \quad \text{ as } c \to \infty.
    \end{align*}
	Similarly
	\begin{equation*}
		\int_{\mathbb{R}^3\times \mathbb{R}^3}
		\frac{u^2(x)u(y)(\ua(y)- u(y))}{|x-y|} \, d x d y \to 0, \quad
		\text{ as } c\to \infty.
\end{equation*}
Hence
\begin{align*}
	&\frac{1}{2m}\|\nabla(\ua-u)\|_{L^2}^2+ \lambda\|\ua - u\|_{L^2}^2
	= \frac{1}{2} \left(d\mathcal{J}_\infty(\ua) -
	d\mathcal{J}_\infty(u) \right)[\ua - u]\\
	&\quad +\int_{\mathbb{R}^3\times \mathbb{R}^3}
	\frac{\ua^2(x)\ua(y)(\ua(y) - u(y))}{|x-y|} \, d x d y
	-\int_{\mathbb{R}^3\times \mathbb{R}^3} \frac{u^2(x)u(y)(\ua(y)-
	u(y))}{|x-y|} \, d x d y\\
	&\quad \to 0, \quad \text{ as } c\to \infty,
\end{align*}
which implies 
\begin{equation*}
	\|\ua- u\|_{H^1}\to 0, \quad as \,\,c\to \infty.
\end{equation*}
It follows from Lemma \ref{Lem:A8}, $u$ is a radially symmetric and
nonnegative action ground state solution to \eqref{1.3}(see
\cite[Lemma 9]{MR2561169}). By the uniqueness of radial and
nonnegative solution to \eqref{1.3}, we have $u= \uai$.
\end{proof}

The following lemma implies that convergence of the
pseudo-relativistic ground state in a low Sobolev norm will lead to
convergence in a high Sobolev norm. This lemma was first proved in
\cite[Proposition 6.1]{MR3734982}, and here we prove it by utilizing
the invertibility of the operator $\La$.
\begin{lemma}
	\label{conv_0}
    For all $s>1$,
	\begin{equation*}
		\|\ua- \uai\|_{H^s} \lesssim \|\ua - \uai\|_{H^{1}} +
		\frac{1}{c^2}.
    \end{equation*}
\end{lemma}

\begin{proof}
	Subtracting equation \eqref{1.2} from equation \eqref{1.3}, we
	obtain
    \begin{equation}\label{3.1}
        \begin{split}
			\La(\ua - \uai) & = (P_\infty(D)-P_c(D))\ua
			+2(|x|^{-1}*\left((\ua-\uai)\uai)\right)(\ua-\uai) \\
			& \quad+ \left(|x|^{-1}*((\ua-\uai)^2)\right)\ua.
        \end{split}
    \end{equation}
	It follows from Lemma \ref{Lem:2.6} and Lemma \ref{Lem:2.1} that
    \begin{equation}\label{3.2}
		\left\| P_{c,1}(D) \ua \right\|_{H^s} = \|(P_c(D)-
		P_\infty(D))\ua\|_{H^s}\lesssim \frac{1}{c^2}.
    \end{equation}
	By Lemma \ref{Lem:2.2}, there holds
    \begin{equation}\label{3.3}
		\left\|2(|x|^{-1}*\left((\ua-\uai)\uai)\right)(\ua-\uai) +
		\left(|x|^{-1}*((\ua-\uai)^2)\right)\ua\right\|_{H^s}\lesssim
		\|\ua-\uai\|_{H^s}^2.
    \end{equation}
    Taking \eqref{3.2} and \eqref{3.3} into \eqref{3.1}, we get
	\begin{equation*}
		\|\La(\ua-\uai)\|_{H^{s}}\lesssim\frac{1}{c^2}+
		\|\ua-\uai\|_{H^s}^{2}.
	\end{equation*}
	Since $\La \colon H_r^{s+2}(\mathbb{R}^3)\to H_r^s(\mathbb{R}^3) $
	is invertible, then
    \begin{align*}
		\|\ua - \uai\|_{H^{s+2}} & = \|\La^{-1}\La(\ua -
		\uai)\|_{H^{s+2}} \\
		&\lesssim \|\La(\ua - \uai)\|_{H^{s}} \lesssim \frac{1}{c^2} +
		\|\ua - \uai\|_{H^s}^{2},
    \end{align*}
    which yields
	\begin{equation*}
		\|\ua- \uai\|_{H^s} \lesssim \|\ua - \uai\|_{H^{1}} +
		\frac{1}{c^2}.
    \end{equation*}
\end{proof}

\begin{proof}[\textbf{Proof of Proposition \ref{Prop:A1}}]
	Lemma \ref{conv_0} shows the convergence of $\ua$ in $H^{s}$ for
	all $s > 0$. The uniqueness of $\ua$ can be derived from the
	uniqueness of $\uai$ and Proposition 3 in \cite{MR2561169}.
\end{proof}

\section{Asymptotic property of  action ground state}
\label{sec:asymptotic_action}
We now come to the main issues of the present paper which focuses on
asymptotic property of action ground state $\ua$ to \eqref{1.2}. 

Lemma \ref{Lem:3.1} yields there exists unique $\fa_1 \in
\bigcap\limits_{s>0} H_r^s(\mathbb{R}^3)$ which satisfy the following
equation
\begin{equation}\label{3.5}
    \La \fa_1 = -P_{\infty,1}(D) \uai.
\end{equation}
Set
\begin{equation}
	\label{eq:deffc1}
	 \fa_{c,1} = c^2(\ua - \uai).
\end{equation}
\begin{lemma}\label{Lem:3.4}
	For each $s>0$, there holds $\|\fa_{c,1} - \fa_{1} \|_{H^s} \to
	0$, as $c\to \infty.$
\end{lemma}
\begin{proof}
    Let's recall elliptic operator
    \begin{equation*}
		P_{c,n}(D)= \sqrt{-c^2\Delta +m^2c^4} - mc^2 - \sum_{k=1}^{n}
		\frac{(-1)^{k-1}\alpha_k}{m^{2k-1}c^{2k-2}}(-\Delta)^k ,
    \end{equation*}
    \begin{equation*}
		P_{\infty,n}(D) :=
		\frac{(-1)^{n}\alpha_{n+1}}{m^{2n+1}}(-\Delta)^{n+1},
    \end{equation*}
    and
    \begin{equation*}
		P_{c}(D) := \sqrt{-c^2\Delta + m^2c^4} - mc^2.
    \end{equation*}
    Then
    \begin{equation*}
        P_{c,n+1}(D) = P_{c,n}(D) - \frac{1}{c^{2n}}P_{\infty, n}(D).
    \end{equation*}
    From  \eqref{1.2} and \eqref{1.3}, we have
    \begin{equation}
        \begin{split}
			-P_{c,1}(D) \ua = \mathcal{N}(\uai) - \mathcal{N}(\ua)+ (
			P_{\infty}(D) + \lambda )(\ua - \uai).
        \end{split}
    \end{equation}
    Multiply both sides of the equation by $c^2$, then
    \begin{equation}\label{3.7}
        \begin{split}
			-c^2P_{c,1}(D) \ua = \La \fa_{c,1} - \Ga_{c,1},
        \end{split}
    \end{equation}
    where
	\begin{equation}
		\label{eq:defGc1}
		\Ga_{c,1} = c^2\left( \mathcal{N}(\uai +
		\frac{1}{c^2}\fa_{c,1}) - \mathcal{N}(\uai)-
		\mathcal{N}^{(1)}(\uai) \left[\frac{1}{c^2}
		\fa_{c,1}\right]\right).
	\end{equation}
    By combining \eqref{3.5} with \eqref{3.7}, we have
    \begin{equation}
		\label{3.8}
		c^2P_{c,1}(D) \ua - P_{\infty,1}(D) \uai = \La(\fa_1 -
		\fa_{c,1}) + \Ga_{c,1}.
    \end{equation}
    By Lemma \ref{Lem:2.4}, there holds
    \begin{equation}
		\label{eq:estimateGc1}
		\|\Ga_{c,1}\|_{H^s} \lesssim c^2 \left \|\frac{\fa_{c,1}}{c^2}
		\right\|_{H^s}^2 = o_{c}(1) \|\fa_{c,1}\|_{H^s} \leq o_{c}(1)
		+ o_{c}(1) \|\fa_{c,1} - \fa_{1}\|_{H^s}.
    \end{equation}
    By Lemma \ref{Lem:2.1} and Lemma \ref{Lem:2.6}, we obtain
    \begin{equation}
		\label{eq:estimatenorm}
        \begin{split}
			\left\|c^2P_{c,1}(D) \ua - P_{\infty,1}(D) \uai
			\right\|_{H^s} & \lesssim \|c^2P_{c,2}(D) \ua \|_{H^s} +
			\|P_{\infty,1}(D)(\ua - \uai)\|_{H^s} \\
			 & \lesssim \frac{1}{c^2}\| \ua\|_{H^{s+6}} +\|\ua -
			 \uai\|_{H^{s+4}} \\
			 & = o_{c}(1),\quad \text{as}\,\,\, c\to \infty,
        \end{split}
    \end{equation}
	then, applying Lemma \ref{Lem:3.3} to \eqref{3.8}, we conclude
	that $\|\fa_{c,1} - \fa_{1}\|_{H^s} = o_{c}(1)$.
\end{proof}

\begin{lemma}
	\label{lem:existencefj}
	For each $j \in \mathbb{N} \cup \{0\}$ there exists a unique
	radial function $\fa_{j} \in \bigcap_{s > 0}
	H^{s}_{r}(\mathbb{R}^3)$ such that for all $n \in \mathbb{N}$ and
	$s > 0$
    \begin{equation}
		\label{3.11}
        \left\|c^{2n}\ua - \sum_{j=0}^{n}c^{2(n-j)} \fa_{j}\right\|_{H^s} \to 0 
		\qquad\text{as } c \to +\infty.
    \end{equation}
	
	In particular $\fa_{0} = \uai$ and $\fa_{k}$ $(k \geq 1)$ is
	the unique radial solution of the following equation
    \begin{equation}
		\label{eq:deffk}
		\La \fa_k = - \sum_{j=0}^{k-1} P_{\infty, k-j}(D) \fa_{ j} +
		\mathcal{T}_k(\fa_0, \fa_{1}, \cdots, \fa_{k-1}),
    \end{equation}
    where
    \begin{equation}\label{3.12}
        \begin{split}
			\mathcal{T}_{k}(\fa_0, \fa_{1}, \fa_{2},\cdots, \fa_{k-1})
			=\sum_{j=2}^{\min\{k,3\}}\frac{1}{j!}
			\sum_{\substack{i_1+i_2+\cdots+i_j = k\\1\leq i_1,
			i_2, \cdots, i_j \leq k-1}}
			\mathcal{N}^{(j)}(\fa_0)[\fa_{i_1}, \fa_{i_2},\cdots, \fa_{i_j}].
        \end{split}      
    \end{equation}
	is nonlinearity of the Hartree-type respect to $\fa_0,
	\fa_{1},\fa_{2},\cdots,\fa_{k-1}$ and $\mathcal{T}_{1}=0$.
\end{lemma}

\begin{proof}	
	Let $\fa_{c,0} = \ua$ and $\fa_{0} = \uai$. We will construct
	successive, higher order (in $\frac{1}{c^{2}}$) approximations of
	$\ua$, which will be denoted by $\fa_{n}$. The difference between
	$\ua$ and its approximation
	$\sum_{j=0}^{n-1}\frac{\fa_j}{c^{2j}}$, once rescaled by $c^{2n}$,
	will be denoted by $\fa_{c,n}$. The proof will be carried out by
	induction on the order of approximation $n$.
	
	We start by defining the approximation terms. Starting with $\fa_{0}
	= \uai$ and assuming we have found $\fa_{j} \in \bigcap_{s >
	0} H^{s}_{r}(\mathbb{R}^{3})$ for $j = 0, \ldots, k-1$ satisfying
	\eqref{eq:deffk} we deduce from Lemma \ref{Lem:3.1} that there
	exists a unique solution $\fa_{k} \in \bigcap_{s > 0}
	H^{s}_{r}(\mathbb{R}^{3})$ to \eqref{eq:deffk}.
		
	We now prove by induction on $n$ that \eqref{3.11} holds. 
	
	Since $\fa_{0} = \uai$ and \eqref{eq:deffk} reduces to
	\eqref{3.5}, we have that $c^{2}\ua - c^{2}\fa_{0} - \fa_{1} =
	\fa_{c,1} - \fa_{1}$ and Lemma \ref{Lem:3.4} implies that
	\eqref{3.11} holds for $n=1$.

	Assuming that for each $1 \leq k \leq n-1$ and for all $s
	> 0$
    \begin{equation*}
		\|c^{2k} \ua - \sum_{j=0}^{k}c^{2(k-j)}\fa_{j} \|_{H^s}\to 0.
    \end{equation*}
	we will show that \eqref{3.11} holds for $n$.

	Recalling the definitions \eqref{eq:deffc1} and \eqref{eq:defGc1}
	of $\fa_{c,1}$ and $\Ga_{c,1}$, we set for each $k \geq 2$
	\begin{equation*}
		\fa_{c,k} = c^2(\fa_{c, k-1} - \fa_{k-1}) = c^{2k} \ua -
		\sum_{j=0}^{k-1} c^{2(k-j)} \fa_{j} = c^{2k} \left( \ua -
		\sum_{j=0}^{k-1}\frac{\fa_j}{c^{2j}} \right)
	\end{equation*}
    and
    \begin{equation*}
		\Ga_{c, k}= c^2\left(\Ga_{c, k-1} - \mathcal{T}_{ k-1}( \fa_0,
		\fa_{1}, \fa_{2}, \cdots, \fa_{k-2})\right).
    \end{equation*}
	With these positions, we have that our induction assumption is 
	that for all $k = 0, \ldots, n-1$ and for all $s > 0$
	\begin{equation*}
		\|\fa_{c, k} - \fa_k\|_{H^s} \to 0 \quad \text{as}\quad c\to 
		\infty.
	\end{equation*}

	It follows from \eqref{3.7} and \eqref{3.5} that $\fa_{c,1}$,
	$\fa_{1}$ satisfies
    \begin{align*}
		&\La \fa_{c,1} = -c^2P_{c,1}(D)\ua + G_{c,1}, \\
		&\La  \fa_1 = -P_{\infty,1}(D) \uai + \mathcal{T}_0,
    \end{align*}
    and, by induction, we deduce that $\fa_{c, k}$ satisfies
    \begin{equation}
		\label{3.14}
		\La \fa_{c, k} = c^2 \left[ -\sum_{j=0}^{k-2}
		c^{2k-2j-2} P_{c,k-j}(D) \fa_{j}- P_{c,1}(D) \fa_{c,k-1}
		\right] + \Ga_{c, k}.
    \end{equation}
    Subtract \eqref{eq:deffk} from equation \eqref{3.14},
    we get
    \begin{multline}
		\La (\fa_{c, n} - \fa_n) = c^2 \left[ -\sum_{j=0}^{n-2}
		c^{2n-2j-2} P_{c,n-j}(D) \fa_{j} - P_{c,1}(D) \fa_{c,n-1}
		\right] \\
		+ \sum_{j=0}^{n-1} P_{\infty, n-j}(D) \fa_{ j} +
		\left(\Ga_{c, n} - \mathcal{T}_{ n}(\fa_0, \fa_{1}, \fa_{2},
		\cdots, \fa_{n-1}) \right).
    \end{multline}
	We claim that, as $c \to +\infty$
    \begin{equation}\label{3.17}
		\| \Ga_{c, n} - \mathcal{T}_{n}(\fa_0, \fa_{1}, \fa_{2},
		\cdots, \fa_{n-1}) \|_{H^s} = o_{c}(1)
	\end{equation} 
	and
    \begin{equation}\label{3.18}
		\left\|c^2\left( -\sum_{j=0}^{n-2} c^{2n-2j-2}
		P_{c,n-j}(D) \fa_{j} - P_{c,1}(D) \fa_{c,n-1} \right) +
		\sum_{j=0}^{n-1} P_{\infty, n-j}(D) \fa_{ j}
		\right\|_{H^s} = o_{c}(1).
    \end{equation}
	We defer the verification of those claims to the end of the proof.
	By combining Lemma \ref{Lem:3.3} with claim \eqref{3.17}, claim
	\eqref{3.18}, we have
	\begin{equation*}
		\|\fa_{c, n} - \fa_n\|_{H^s} \to 0 \quad \text{as}\quad c\to
		\infty,
	\end{equation*}
    that is
    \begin{equation*}
		\|c^{2n} \ua- \sum_{j=0}^{n}c^{2(n-j)} \fa_{j}\|_{H^s}\to 0.
    \end{equation*}

	\subsection*{Proof of the Claim \eqref{3.17}} From
    \begin{equation}\label{3.13}
        \begin{split}
			\Ga_{c, k} & = c^{2k} \left(\frac{\Ga_{c,1}}{c^2} -
			\sum_{j=2}^{k-1} \frac{\mathcal{T}_{ j}(\fa_0, \fa_{1},
			\fa_{2}, \cdots, \fa_{j-1})}{c^{2j}} \right) \\
			& = c^{2k}\left(\mathcal{N}(\ua) - \mathcal{N}(\fa_0)-
			\mathcal{N}^{(1)}(\fa_0)[\ua - \fa_0]- \sum_{j=2}^{k-1}
			\frac{\mathcal{T}_{ j}(\fa_0, \fa_{1}, \fa_{2}, \cdots,
			\fa_{j-1})}{c^{2j}} \right).
        \end{split}
    \end{equation}
	and the triangle inequality, we obtain
	\begin{equation}\label{3.19}
		\begin{split}
			& \|\Ga_{c, n} - \mathcal{T}_{ n}(\fa_0, \fa_{1}, \fa_{2},
			\cdots, \fa_{n-1}) \|_{H^s} \\
			&= c^{2n} \left\|\mathcal{N} \left( \ua \right) -
			\mathcal{N}(\fa_0) - \mathcal{N}^{(1)}(\fa_0) \left[ \ua
			-\fa_0 \right] - \sum_{j=2}^{n} \frac{\mathcal{T}_{
			j}(\fa_0, \fa_{1}, \fa_{2}, \cdots, \fa_{j-1})}{c^{2j}}
			\right\|_{H^s} \\
			& \leq c^{2n} \left\| \mathcal{N} \left( \ua \right)
			- \mathcal{N}(f_0) - \mathcal{N}^{(1)}(\fa_0) \left[\ua -
			\fa_0\right] - \sum_{j=2}^{\min\{n,3\}} \frac{1}{j!}
			\mathcal{N}^{(j)}(\fa_0)[ \underbrace{\ua - \fa_0, \cdots,
			\ua - \fa_0}_{\text{$j$ times}} ] \right\|_{H^s} \\
			&\quad + c^{2n} \left\| \sum_{j=2}^{\min\{n,3\}}
			\frac{1}{j!}\mathcal{N}^{(j)}(\fa_0)[\underbrace{\ua -
			\fa_0, \cdots, \ua - \fa_0}_{\text{$j$ times}}] -
			\sum_{j=2}^{n} \frac{\mathcal{T}_{ j}(\fa_0, \fa_{1},
			\fa_{2}, \cdots, \fa_{j-1})}{c^{2j}} \right\|_{H^s}\\
			&=: I_1 + I_2.
		\end{split}
	\end{equation}
    By Lemma \ref{Lem:2.4}, $I_1=0$ provided $n\geq 3$. For $n=2$,
    Lemma \ref{Lem:3.4} implies
    \begin{equation}\label{3.20}
		\norm{\fa_{c,0} - \fa_{0}}_{H^{s}} = \|\ua - \fa_0\|_{H^s}
		\lesssim \frac{1}{c^2},
    \end{equation}
    by combining Lemma \ref{Lem:2.4} with \eqref{3.20}, we obtain
	\begin{equation}\label{3.21}
		\begin{split}
			\qquad\quad I_1 = c^{4} & \left\| \mathcal{N} \left( \ua
			\right) - \mathcal{N}(\fa_0) - \mathcal{N}^{(1)}(\fa_0)
			\left[ \ua - \fa_0 \right] -\frac{1}{2}
			\mathcal{N}^{(2)}(\fa_0)[ \ua - \fa_0, \ua - \fa_0]
			\right\|_{H^s}\\
			&\leq C c^{4}\| \ua - \fa_0 \|_{H^s}^{3}\lesssim
			\frac{1}{c^2}.
        \end{split}
    \end{equation}
	From the assumption: $\| \fa_{c, n-1} - \fa_{n-1} \|_{H^s} =
	o_{c}(1)$, we have
	\begin{equation*}
		\ua - \fa_0 = \frac{\fa_{c, n-1}}{c^{2(n-1)}} +
		\sum_{i=1}^{n-2} \frac{\fa_i}{c^{2i}} = \sum_{i=1}^{n-1}
		\frac{\fa_i}{c^{2i}} + o_{c}(\frac{1}{c^{2n-2}}).
	\end{equation*}
    Consequently,
    \begin{equation}
        \begin{split}
			 & \mathcal{N}^{(j)}(f_0)[\underbrace{\ua -
			 \fa_0,\cdots,\ua - \fa_0}_{j \,\,\text{times}}] \\
			& \quad= \sum_{\substack{1\leq i_1, i_2, \cdots, i_j\leq
			n-1}}\frac{1}{c^{2(i_1+ i_2+\cdots +
			i_j)}}\mathcal{N}^{(j)}(\fa_0)[\fa_{i_1},
			\fa_{i_2},\cdots, \fa_{i_j}] + o_{c}(\frac{1}{c^{2n}}) \\
             & \quad
            =\sum_{\substack{i_1+ i_2+ \cdots+i_j\leq n                                                                                                                                  \\
			1\leq i_1, i_2, \cdots, i_j\leq n-1}}\frac{1}{c^{2(i_1+
			i_2+\cdots + i_j)}}\mathcal{N}^{(j)}(\fa_0)[\fa_{i_1},
			\fa_{i_2},\cdots, \fa_{i_j}] + o_{c}(\frac{1}{c^{2n}}),
        \end{split}
    \end{equation}
    and
    \begin{equation}\label{3.23}
        \begin{split}
			\qquad & \sum_{j=2}^{\min\{n,3\}} \frac{1}{j!}
			\mathcal{N}^{(j)}(\fa_0)[\underbrace{\ua -
			\fa_0,\cdots,\ua - \fa_0}_{j \,\,\text{times}}] \\
			 & \quad=
			 \sum_{j=2}^{\min\{n,3\}}\frac{1}{j!}\sum_{\substack{i_1+
			 i_2+ \cdots+i_j\leq n \\1\leq i_1, i_2, \cdots, i_j\leq
			 n-1}}\frac{1}{c^{2(i_1+ i_2+\cdots +
			 i_j)}}\mathcal{N}^{(j)}(\fa_0)[\fa_{i_1},
			 \fa_{i_2},\cdots, \fa_{i_j}] + o_{c}(\frac{1}{c^{2n}}).
        \end{split}
    \end{equation}
    Then, from \eqref{3.12} and the fact that
	\begin{equation*}
		\sum_{j=2}^{n}\sum_{k=2}^{\min\{j,3\}}a_{k,j}=
		\sum_{k=2}^{\min\{n,3\}}\sum_{j=k}^na_{k,j},
	\end{equation*}
    we get
    \begin{equation}\label{3.24}
		\begin{split}
			\sum_{j=2}^{n}\frac{\mathcal{T}_{ j}(\fa_0,
			\fa_{1},\fa_{2},\cdots,\fa_{j-1})}{c^{2j}} & =
			\sum_{j=2}^{n} \sum_{k=2}^{\min\{j,3\}} \frac{1}{k!}
			\frac{1}{c^{2j}} \sum_{\substack{i_1+i_2+\cdots+i_k=j
			\\i_1, i_2, \cdots, i_k \geq 1}} \mathcal{N}^{(k)}(\fa_0)
			[\fa_{i_1}, \fa_{i_2},\cdots, \fa_{i_k}] \\
			& = \sum_{k=2}^{\min\{n,3\}} \sum_{j=k}^n \frac{1}{k!}
			\frac{1}{c^{2j}} \sum_{\substack{i_1+i_2+\cdots+i_k=j \\
			i_1, i_2, \cdots, i_k \geq 1}} \mathcal{N}^{(k)}(\fa_0)
			[\fa_{i_1}, \fa_{i_2},\cdots, \fa_{i_k}] \\
			& = \sum_{j=2}^{\min\{n,3\}} \frac{1}{j!}
			\sum_{\substack{i_1+ i_2+ \cdots+i_j\leq n \\ 1\leq i_1,
			i_2, \cdots, i_j\leq n-1}}\frac{1}{c^{2(i_1+ i_2+\cdots +
			i_j)}}\mathcal{N}^{(j)}(\fa_0)[\fa_{i_1},
			\fa_{i_2},\cdots, \fa_{i_j}].
        \end{split}
    \end{equation}
    Combining \eqref{3.23} with \eqref{3.24}, there holds
    \begin{equation}\label{3.25}
		c^{2n}\left\| \sum_{j=2}^{\min\{n,3\}} \frac{1}{j!}
		\mathcal{N}^{(j)}(\fa_0)[\underbrace{\ua - \fa_0,\cdots,\ua -
		\fa_0}_{j \,\,\text{times}}] -
		\sum_{j=2}^{n}\frac{\mathcal{T}_{ j}(\fa_0,
		\fa_{1},\fa_{2},\cdots,\fa_{j-1})}{c^{2j}} \right\|_{H^s} =
		o_{c}(1).
    \end{equation}
    By taking \eqref{3.25}, \eqref{3.21} into \eqref{3.19}, we get
    \begin{equation}\label{jjj_1}
		\|G_{c, n} - \mathcal{T}_{ n}(\fa_0,
		\fa_{1},\fa_{2},\cdots,\fa_{n-1})\|_{H^s} \to 0, \quad
		\text{as}\quad c\to \infty.
    \end{equation}
    which yields the Claim \eqref{3.17} holds.

    \subsection*{Proof of the Claim \eqref{3.18}}
    Since
    \begin{equation*}
        \begin{split}
			c^2 & \left(
			-\sum_{j=0}^{n-2}c^{2n-2j-2}P_{c,n-j}(D)\fa_{j}-
			P_{c,1}(D)\fa_{c,n-1}
			\right)+\sum_{j=0}^{n-1}P_{\infty, n-j}(D)\fa_{ j}
			\\
			& = -\sum_{j=0}^{n-1} c^{2n-2j} P_{c,n-j+1}(D) \fa_{j}-
			c^2P_{c,1}(D)(\fa_{c, n-1}-\fa_{n-1}),
        \end{split}
    \end{equation*}
    then by Lemma \ref{Lem:2.1},
    \begin{equation}
        \begin{split}
			 & \left\|c^2\left( - \sum_{j=0}^{n-2} c^{2n-2j-2}
			 P_{c,n-j}(D) \fa_{j}- P_{c,1}(D)\fa_{c,n-1}
			 \right)+\sum_{j=0}^{n-1}P_{\infty, n-j}(D)\fa_{
			 j}\right\|_{H^s} \\
			 & \quad \lesssim \sum_{j=0}^{n-1}c^{2n-2j}
			 \|P_{c,n-j+1}(D)\fa_{j}\|_{H^s}+ c^2 \|P_{c,1}(D)(\fa_{c,
			 n-1}-\fa_{n-1})\|_{H^s} \\
			 & \quad\lesssim \frac{1}{c^2} \sum_{j=0}^{n-1}
			 \|\fa_j\|_{H^{s+ 2(n-j+2)}}+ \|\fa_{c,
			 n-1}-\fa_{n-1}\|_{H^{s+4}} \\
			 & \quad = o_{c}(1), \quad \text{as}\quad c\to
			 \infty.
        \end{split}
    \end{equation}
    Hence, the Claim \eqref{3.18} is verified.

\end{proof}

\section{Asymptotic property of  energy ground state}
In this section, we focus on asymptotic property of energy ground
state $\ue$ to \eqref{1.8}. Denote the unique positive radial energy
ground state to \eqref{1.9} by $\uei $. Then $\ue$, $\uei $
satisfy \eqref{1.10} (resp., \eqref{1.11}) with Lagrange multiplier
$\omega_c$ ($\omega_\infty$) $\in \mathbb{R}$.

The linearized operator of \eqref{1.11} at $\uei $ is given by
\begin{equation*}
    \begin{split}
		\Le u: & = \left(P_\infty(D)+ \omega_\infty\right)u-
		\mathcal{N}^{(1)}(\uei )u \\
		& = -\frac{1}{2m}\Delta u +\omega_\infty u-
		\left(|x|^{-1}*\uei ^2\right)u-2\left(|x|^{-1}*(\uei
		u)\right)\uei .
    \end{split}
\end{equation*}
Thanks to the non-degeneracy of $\uei$, Lemma \ref{Lem:3.1} and Lemma
\ref{Lem:3.3} still holds for $\Le$.

We know from Theorem \ref{thm:energyGS} that $\ue  \in H^{s}$ for all
$s \geq 0$ and that $\ue  \to \uei $ in $H^{1}(\mathbb{R}^{3})$
as $c \to +\infty$.

As in Section \ref{sec:action} we can show that $\ue  \to \uei $
in $H^{s}$ for all $s \geq 0$, that is also for the energy ground
states $\ue $ Lemma \ref{conv_0} holds:
\begin{lemma}
	\label{conv_energy}
	For all $s \geq 0$,
	\begin{equation*}
		\|\ue  - \uei \|_{H^s} \to 0.
    \end{equation*}
\end{lemma}

\begin{proof}
	We proceed as in Lemma \ref{conv_0}: subtracting equation
	\eqref{1.10} from equation \eqref{1.11}, we obtain
    \begin{multline}
		\label{eq:5.1}
		\Le (\ue -\uei ) = (P_\infty(D)-P_c(D))\ue  +
		(\omega_{\infty} - \omega_{c})\ue  \\
	+ 2(|x|^{-1}*\left((\ue  - \uei )\uei )\right)(\ue -\uei )
	+ \left(|x|^{-1}*((\ue -\uei )^2)\right)\ue . \end{multline}
	It follows from Lemma \ref{Lem:2.6} (which holds also for energy
	ground states) and Lemma \ref{Lem:2.1} that
    \begin{equation}\label{eq:5.2}
		\left\| P_{c,1}(D) \ue  \right\|_{H^s} = \|(P_c(D)-
		P_\infty(D))\ue \|_{H^s} \lesssim \frac{1}{c^2} \quad \text{
		and } \quad \abs{\omega_{\infty} - \omega_{c}}
		\norm{\ue }_{H^{s}} \lesssim \abs{\omega_{\infty} -
		\omega_{c}}.
    \end{equation}
    By Lemma \ref{Lem:2.2}, there holds
    \begin{equation}\label{eq:5.3}
		\left\|2(|x|^{-1}*\left((\ue -\uei )\uei )\right)(\ue -\uei )
		+
		\left(|x|^{-1}*((\ue -\uei )^2)\right)\ue \right\|_{H^s}\lesssim
		\|\ue -\uei \|_{H^s}^2.
    \end{equation}
    Taking \eqref{eq:5.2} and \eqref{eq:5.3} into \eqref{eq:5.1}, we get
	\begin{equation*}
		\|\Le (\ue -\uei )\|_{H^{s}} \lesssim \frac{1}{c^2}
		+ \abs{\omega_{\infty} - \omega_{c}} +
		\|\ue -\uei \|_{H^s}^{2}.
	\end{equation*}
	Since $\Le : H_r^{s+2}(\mathbb{R}^3)\to H_r^s(\mathbb{R}^3)
	$ is invertible, then
    \begin{align*}
		\|\ue -\uei \|_{H^{s+2}} &
		=\|\Le ^{-1}\Le (\ue -\uei )\|_{H^{s+2}} \\
		&\lesssim \|\Le (\ue -\uei )\|_{H^{s}} \lesssim
		\frac{1}{c^2} + \abs{\omega_{\infty} - \omega_{c}} +
		\|\ue -\uei \|_{H^s}^{2},
    \end{align*}
    which yields
	\begin{equation*}
		\|\ue - \uei \|_{H^s} \lesssim \|\ue  -
		\uei \|_{H^{1}} + \abs{\omega_{\infty} - \omega_{c}}
		+\frac{1}{c^2}.
    \end{equation*}
	Since $\ue  \to \uei $ in $H^{1}$ the result follows.
\end{proof}

To understand the asymptotic behavior for the energy ground state, we
study, as in section \ref{sec:asymptotic_action}, the behavior of
\begin{equation}
	\label{new_q10}
	\fe_{c,1}= c^2(\ue  - \uei ),
\end{equation} 
see \eqref{eq:deffc1} and Lemma \ref{Lem:3.4}.

Multiplying both sides of equation \eqref{eq:5.1} by $c^2$, we get the
following equation for $\fe_{c,1}$:
\begin{equation}\label{4.1}
    \begin{split}
		\Le \fe_{c,1} & = c^2(P_\infty(D)-P_c(D))\ue +
		c^2(\omega_\infty- \omega_c)\ue  \\
		& \quad+2(|x|^{-1}*(\fe_{c,1}\uei ))(\ue -\uei )+
		\left(|x|^{-1}*(\fe_{c,1}(\ue -\uei ))\right)\ue .
    \end{split}
\end{equation}

To analyze such an equation and prove the analogue of Lemma
\ref{Lem:3.4} we need to study the asymptotic behavior of the Lagrange
multiplier $\omega_c$. The next few lemmas are devoted to this task.

\begin{lemma}\label{Lem:4.2}
    $\lim\limits_{c\to\infty}c^2(e_\infty- e_c)= \frac{1}{8m^3}\|\Delta \uei \|_{L^2}^2$.
\end{lemma}
\begin{proof}
    By Lemma \ref{Lem:4.1}, there holds
	\begin{equation*}
		P_\infty(D)-\frac{1}{8m^3c^2}(-\Delta)^2\leq P_c(D)\leq
		P_\infty(D)-\frac{1}{8m^3c^2}(-\Delta)^2 +
		\frac{1}{16m^5c^4}(-\Delta)^3.
	\end{equation*}
    Then
    \begin{equation}\label{4.5}
        \begin{split}
			e_c & = \mathcal{E}_{c}(\ue ) = \scalar{P_c(D)\ue }{\ue } -
			\frac{1}{2} \mathcal{H}(\ue ) \\
			& \geq \scalar{P_\infty(D)\ue }{\ue } - \frac{1}{2}
			\mathcal{H}(\ue ) -
			\frac{1}{8m^3c^2}\scalar{(-\Delta)^2\ue }{\ue } \\
			& \geq e_\infty - \frac{1}{8m^3c^2} \|\Delta
			\ue \|_{L^2}^2.
        \end{split}
    \end{equation}
    On the other hand,
    \begin{equation}\label{4.6}
        \begin{split}
			e_\infty & = \mathcal{E}_{\infty}(\uei ) =
			\scalar{P_\infty(D)\uei }{\uei } - \frac{1}{2}
			\mathcal{H}(\uei ) \\
			& \geq \scalar{P_c(D)\uei }{\uei } -
			\frac{1}{2} \mathcal{H}(\uei ) \\&
			\quad+\frac{1}{8m^3c^2}\|\Delta
			\uei \|_{L^2}^2-\frac{1}{16m^5c^4}\|(-\Delta)^\frac{3}{2}
			\uei \|_{L^2}^2 \\& \geq e_c +
			\frac{1}{8m^3c^2}\|\Delta
			\uei \|_{L^2}^2-\frac{1}{16m^5c^4}\|(-\Delta)^\frac{3}{2}
			\uei \|_{L^2}^2.
        \end{split}
    \end{equation}
    By combining \eqref{4.5} and \eqref{4.6}, we get
    \begin{equation*}
	\lim_{c\to\infty}c^2(e_\infty- e_c)= \frac{1}{8m^3}\|\Delta
	\uei \|_{L^2}^2.
    \end{equation*}
\end{proof}

The following results were proved in \cite[Lemma 3.2]{MR4078531}.
\begin{lemma}\label{Lem:4.3}
    $\ue $ and $\uei $ satisfy the  Pohozaev identity
    \begin{equation}\label{4.7}
		\scalar{T_c(D)\ue }{\ue } + e_c = 0,
    \end{equation}
    and
    \begin{equation}\label{4.8}
		\scalar{P_\infty(D)\uei }{\uei } + e_\infty = 0,
    \end{equation}
    where
	\begin{equation*}
		T_c(D)=\frac{P_c(D)}{\sqrt{-\Delta/(m^2c^2)+1}}.
	\end{equation*}
\end{lemma}
\begin{proof}
    We only prove \eqref{4.8} here.
    A direct computation shows,
    \begin{multline*}
        T_c(D)- P_\infty(D) \\
		= \frac{(P_c(D)- P_\infty(D))}{\sqrt{-\Delta/(m^2c^2)+1}} -
		\frac{(-\Delta)^2}{2m^3c^2\sqrt{-\Delta/(m^2c^2) +
		1}(\sqrt{-\Delta/(m^2c^2)+1} + 1)},
    \end{multline*}
	then by the fact that $\ue \to \uei $ in $H^s$ for each $s \geq
	0$, we have
    \begin{equation}
        \begin{split}
			\| \left(T_c(D)-P_\infty(D)\right)\ue \|_{L^2} & \lesssim
			\|(P_c(D)- P_\infty(D))\ue \|_{L^2}+
			\frac{1}{c^2}\|\Delta^2\ue \|_{L^2} \\& = o_{c}(1),
			\quad \text{as}\quad c\to \infty.
        \end{split}
    \end{equation}
    and
    \begin{equation}\label{4.10}
        \begin{split}
			\qquad\,\,\,\| T_c(D)\ue -P_\infty(D)\uei \|_{L^2} & \leq
			\| \left(T_c(D)-P_\infty(D)\right)\ue \|_{L^2}+ \|
			P_\infty(D)(\ue -\uei )\|_{L^2} \\& = o_{c}(1), \quad
			\text{as}\quad c\to \infty.
        \end{split}
    \end{equation}
	Taking \eqref{4.10} into \eqref{4.7}, and considering the limit as
	$e_c \to e_\infty$, we arrive that
	\begin{equation*}
		0= \lim_{c\to \infty}\scalar{T_c(D)\ue }{\ue } + e_c =
		\scalar{P_\infty(D)\uei }{\uei } + e_\infty.
	\end{equation*}
\end{proof}
\begin{lemma}\label{Lem:4.4}
	$ \lim\limits_{c\to \infty}c^2(\omega_\infty - \omega_c) =
	-\frac{5}{8m^3}\|\Delta \uei \|_{L^2}^2.$
\end{lemma}
\begin{proof}
	Lemma \ref{Lem:4.1} implies that 
    \begin{equation*}
			c^2(P_c(D)- P_\infty(D)) \to -
			\frac{1}{8m^{2}}(-\Delta)^{2}.
    \end{equation*}

    It follows from Lemma \ref{Lem:4.3} that
    \begin{equation}\label{4.11}
        \begin{split}
			c^2(e_\infty - e_c) & = c^2 \left(\scalar{T_c(D)\ue }{\ue } -
			\scalar{P_\infty(D)\uei }{\uei } \right)\\
			&= \scalar{c^2(T_c(D)- P_\infty(D))\ue }{ \ue } +
			\scalar{P_\infty(D)\fe_{c,1}}{\ue  + \uei }.
        \end{split}
    \end{equation}
	and
    \begin{multline}\label{4.12}
            c^2(T_c(D)- P_\infty(D)) 
            \\
			=\frac{c^2(P_c(D)- 
			P_\infty(D))}{\sqrt{-\Delta/(m^2c^2)+1}} 
			-\frac{(-\Delta)^2}{2m^3\sqrt{-\Delta/(m^2c^2)+1}(\sqrt{-\Delta/(m^2c^2)+1}+1)}
			\\
			\to -\frac{3\Delta^2}{8m^3}, \quad \text{as}\quad c\to \infty,
    \end{multline}
    then it follows from $\|\ue -\uei \|_{H^s}\to 0$ that
    \begin{equation}\label{4.13}
		\lim_{c\to\infty} \scalar{c^2(T_c(D)-
		P_\infty(D))\ue }{\ue } = -\frac{3}{8m^3}\|\Delta
		\uei \|_{L^2}^2.
    \end{equation}
    Since
	\begin{equation*}
		-\omega_c = 2e_c - \scalar{P_c(D)\ue }{\ue },
	\end{equation*}
	and
	\begin{equation*}
		-\omega_\infty = 2e_\infty - \scalar{P_\infty(D)\uei }{\uei },
	\end{equation*}
    then by \eqref{4.11}, there holds
    \begin{multline}\label{4.14}
		c^{2}(\omega_\infty - \omega_c) \\
		= 2c^{2} (e_c - e_\infty) + c^{2} \scalar{(P_\infty(D) -
		P_c(D))\ue }{\ue } - \scalar{(P_\infty(D)\fe_{c,1}}{\ue + \uei
		} \\
		= 3c^2 (e_c - e_\infty) + c^2 \scalar{(P_\infty(D) -
		P_c(D))\ue }{\ue } + c^{2}\scalar{(T_c(D)-
		P_\infty(D))\ue }{\ue }.
    \end{multline}
    By combining Lemma \ref{Lem:4.2} with \eqref{4.13}, we obtain
	\begin{equation*}
		\lim_{c \to \infty} c^2(\omega_\infty-\omega_c) =
		-\frac{5}{8m^3}\|\Delta \uei \|_{L^2}^2.
	\end{equation*}
\end{proof}
The non-degeneracy of the linearized operator $\Le $ yields
the following linearized equation has only the unique radial solution
$\fe_1\in H^1_r$

\begin{equation}\label{4.15}
	\Le \fe_1= -P_{\infty,1}(D)\uei -b_1\uei ,
\end{equation}
where $b_1= \frac{5}{8m^3} \|\Delta \uei \|_{L^2}^2$.

\begin{lemma}
	\label{Lem:4.5}
	For each $s>0$, there holds $\|\fe_{c,1}-\fe_{1}\|_{H^s}\to 0$, as
	$c\to \infty.$
\end{lemma}

\begin{proof}
	The proof follows the lines of the proof of Lemma \ref{Lem:3.4}, 
	but we now have an additional term related to $\omega_{c}$.
	
	From \eqref{4.1} we obtain
	\begin{equation}
		\label{new_q11}
		\Le \fe_{c,1} = -c^2 P_{c,1}(D) \ue  - b_{c,1} \ue  +
		\Ge_{c,1},
	\end{equation}
	where
	\begin{equation}
		\Ge_{c,1} = c^2 \left( \mathcal{N}\left(\uei  + \frac{1}{c^2}
		\fe_{c,1}\right) - \mathcal{N}(\uei ) -
		\mathcal{N}^{(1)}(\uei )\left[\frac{1}{c^2}\fe_{c,1}\right]
		\right),
	\end{equation}
	(see also \eqref{eq:defGc1}) and
	\begin{equation}
		b_{c,1} = c^2 (\omega_c - \omega_\infty).
	\end{equation}
	Subtracting \eqref{4.15} from \eqref{new_q11} yields
	\begin{equation}
		\label{4.16}
		\Le (\fe_{c,1} - \fe_1) = -c^2 P_{c,1}(D) \ue  +
		P_{\infty,1}(D) \uei  - b_{c,1} \ue  + b_1 \uei  +
		\Ge_{c,1}.
	\end{equation}
	Following similar estimates as in
	\eqref{eq:estimateGc1} and \eqref{eq:estimatenorm}, we have
	\begin{equation}
		\left\| -c^2 P_{c,1}(D) \ue  + P_{\infty,1}(D) \uei  +
		\Ge_{c,1} \right\|_{H^s} \lesssim o_c(1) + o_c(1) \| \fe_{c,1} -
		\fe_{1} \|_{H^s}.
	\end{equation}
	Then, by Lemma \ref{Lem:4.4} and the fact that
	\begin{equation*}
		\| \ue  - \uei  \|_{H^s} \to 0 \quad \text{as } c \to \infty,
	\end{equation*}
	we obtain
	\begin{multline*}
		\left\| -c^2 P_{c,1}(D) \ue  + P_{\infty,1}(D) \uei  - 
		b_{c,1} \ue  + b_1 \uei  + \Ge_{c,1} \right\|_{H^s} \\
		\lesssim o_c(1) + o_c(1) \| \fe_{c,1} - \fe_{1}\|_{H^s}.
	\end{multline*}
	Therefore, applying Lemma \ref{Lem:3.3} gives
	\begin{equation*}
		\| \fe_{c,1} - \fe_1 \|_{H^s} = o_c(1).
	\end{equation*}
\end{proof}

Now, we proceed to prove Theorem \ref{them:1.5} by induction. 
First, we need to prove the asymptotic properties of the ground state 
energy $e_c$ and the Lagrange multiplier $\omega_c$.  

\begin{lemma}\label{newlemma:1}
	Assume that for $n \in \mathbb{N}$ there exist radial functions
	$\fe_{1}, \ldots, \fe_n \in \bigcap_{s > 0} H^{s}(\mathbb{R}^3)$
	such that
	\begin{equation*}
		\left\| c^{2n}(\ue  - \uei ) - \sum_{j=1}^{n} c^{2(n-j)}
		\fe_j \right\|_{H^s} \to 0,\quad \text{as} \quad c \to \infty,
	\end{equation*}
	or equivalently,
	\begin{equation}\label{new_q1}
		\ue  = \uei  + \sum_{j=1}^{n} \frac{\fe_j}{c^{2j}} +
		\frac{o_c(1)}{c^{2n}},
	\end{equation}
	then there exist constants $a_j \in \mathbb{R}$, $j = 1, \dots,
	n+1$, such that
	\begin{equation}
		\label{eq:estimateec}
		e_c = e_\infty + \sum_{j=1}^{n+1} \frac{a_j}{c^{2j}} +
		\frac{o_c(1)}{c^{2n+2}}.
	\end{equation}
\end{lemma}

\begin{proof}
	From assumption \eqref{new_q1} and the asymptotic expansion
	\begin{equation}\label{new_q5}
		\sqrt{-c^2\Delta + m^2 c^4} - m c^2 = \sum_{j=0}^{n+1}
		\frac{(-1)^j \alpha_{j+1}}{m^{2j+1} c^{2j}} (-\Delta)^{j+1} +
		\mathcal{O}\left(\frac{1}{c^{2n+4}}\right),
	\end{equation}
	we obtain
	\begin{equation}
		\label{new_q2}
		e_c = \mathcal{E}_c(\ue ) = \mathcal{E}_\infty(\ue ) +
		\sum_{j=1}^{n+1} \frac{(-1)^j \alpha_{j+1}}{m^{2j+1} c^{2j}}
		\left\| (-\Delta)^{\frac{j+1}{2}} \ue \right\|_{L^2}^2 +
		\mathcal{O}\left(\frac{1}{c^{2n+4}}\right)
	\end{equation}
	
	Since for all $j = 1, \ldots, n$ we have that
	\begin{equation*}
		\left\| (-\Delta)^{\frac{j+1}{2}} \left( \uei + \sum_{k=1}^{n}
		\frac{\fe_k}{c^{2k}} + \frac{o_c(1)}{c^{2n}}\right)
		\right\|_{L^2}^2 
		= \sum_{k = 0}^{n} \frac{\beta_{j,k}}{c^{2k}} +
		\frac{o_{c}(1)}{c^{2n+2}}
	\end{equation*}
	for some constant $\beta_{j,k}$ independent of $c$ we deduce that 
	\begin{equation*}
		e_c = \mathcal{E}_\infty(\ue ) + \sum_{j=1}^{n+1}
		\frac{a_{1,j}}{c^{2j}} + \frac{o_c(1)}{c^{2n+2}}
	\end{equation*}
	where, letting $\fe_{0} = \uei$,
	\begin{equation*}
		a_{1,j} = \sum_{\substack{t + s = j \\ t \geq 1}} \sum_{k + 
		\ell = s} \frac{(-1)^{t} \alpha_{t+1}}{m^{2t+1}} 
		\scalar{(-\Delta)^{\frac{t+1}{2}} 
		\fe_{k}}{(-\Delta)^{\frac{t+1}{2}} \fe_{\ell}} \in \mathbb{R}
	\end{equation*}
	are constants independent of $c$.
	By Lemma~\ref{Lem:2.4_new}, we expand
	\begin{equation}\label{new_q3}
		\begin{split}
			\mathcal{E}_\infty(\ue ) &= \sum_{k=0}^4 \frac{1}{k!} \,
			d^k \mathcal{E}_\infty(\uei )\big[ \underbrace{\ue  -
			\uei , \dots, \ue  - \uei }_{k \text{ times}} \big]
			\\
			&= e_\infty - 2\omega_\infty \scalar{\uei }{ \ue  -
			\uei } + \sum_{k=2}^4 \frac{1}{k!} \, d^k
			\mathcal{E}_\infty(\uei )\big[ \underbrace{\ue  -
			\uei , \dots, \ue  - \uei }_{k \text{ times}} \big]
			\\
			&= e_\infty + \omega_\infty \| \ue  - \uei  \|_{L^2}^2 +
			\sum_{k=2}^4 \frac{1}{k!} \, d^k
			\mathcal{E}_\infty(\uei )\big[ \underbrace{\ue  -
			\uei , \dots, \ue  - \uei }_{k \text{ times}} \big].
		\end{split}
	\end{equation}

	From our assumption \eqref{new_q1}, it follows that
	\begin{equation}
		\label{new_q4}
		\omega_\infty \| \ue  - \uei  \|_{L^2}^2 + \sum_{k=2}^4
		\frac{1}{k!} \, d^k \mathcal{E}_\infty(\uei )\big[
		\underbrace{\ue  - \uei , \dots, \ue  - \uei }_{k \text{
		times}} \big] = \sum_{j=1}^{n+1} \frac{a_{2,j}}{c^{2j}} +
		\frac{o_c(1)}{c^{2n+2}},
	\end{equation}
	with
	\begin{multline*}
		a_{2,j} = \sum_{\substack{i + k = j \\ i, j \geq 1}}
		\left(\omega_\infty\scalar{\fe_{i}}{\fe_{k}} + \frac{1}{2}d^2 \mathcal{E}_\infty(\uei
		) [\fe_{i}, \fe_{k}]\right) \\
		+ \frac{1}{6}\sum_{\substack{i + k + \ell = j \\ i, j, \ell \geq 1}} d^3
		\mathcal{E}_\infty(\uei) [\fe_{i}, \fe_{k}, \fe_{\ell}] +\frac{1}{24}
		\sum_{\substack{i + k + \ell + t = j \\ i, j, \ell, t \geq 1}}
		d^4 \mathcal{E}_\infty(\uei) [\fe_{i}, \fe_{k}, \fe_{\ell},
		\fe_{t}].
	\end{multline*}
	Combining \eqref{new_q2}, \eqref{new_q3}, and \eqref{new_q4}, we conclude
	\begin{equation*}
		e_c = e_\infty + \sum_{j=1}^{n+1} \frac{a_j}{c^{2j}} +
		\frac{o_c(1)}{c^{2n+2}},
	\end{equation*}
	where $a_j = a_{1,j} + a_{2,j}$ are constants idependent of $c$.
\end{proof}

\begin{lemma}\label{newlemma:2}
	Assume that for $n \in \mathbb{N}$ there exist radial functions
	$\fe_{1}, \ldots, \fe_n \in \bigcap_{s > 0} H^{s}(\mathbb{R}^3)$
	such that
	\begin{equation*}
		\left\| c^{2n}(\ue - \uei ) - \sum_{j=1}^{n} c^{2(n-j)} \fe_j
		\right\|_{H^s} \to 0, \quad \text{as} \quad c \to \infty,
	\end{equation*}
	or equivalently,
	\begin{equation}\label{new_q6}
		\ue = \uei + \sum_{j=1}^{n} \frac{\fe_j}{c^{2j}} +
		\frac{o_c(1)}{c^{2n}},
	\end{equation}
	then there exist constants $b_j \in \mathbb{R}$, $j = 1, \dots,
	n+1$, such that
	\begin{equation*}
		\omega_c = \omega_\infty + \sum_{j=1}^{n+1} \frac{b_j}{c^{2j}}
		+ \frac{o_c(1)}{c^{2n+2}}.
	\end{equation*}
\end{lemma}

\begin{proof}
	By \eqref{4.14} we have that
	\begin{equation*}
		\omega_c = \omega_\infty - 3 (e_c - e_\infty) -
		\scalar{(P_\infty(D) - P_c(D)) \ue }{\ue } - \scalar{(T_c(D) -
		P_\infty(D)) \ue }{\ue }.
	\end{equation*}
	Let us analyze the different terms.

	From \eqref{new_q5} and assumption \eqref{new_q6}, we have that
	\begin{equation*}
		\scalar{(P_c(D) - P_\infty(D)) \ue }{\ue } = \sum_{j=1}^{n+1}
		\frac{(-1)^j \alpha_{j+1}}{m^{2j+1} c^{2j}} \left\|
		(-\Delta)^{\frac{j+1}{2}} \ue \right\|_{L^2}^2 +
		\mathcal{O}\left(\frac{1}{c^{2n+4}}\right) \\
	\end{equation*}
	and we deduce, as for \eqref{new_q2}
	\begin{equation}
		\label{new_q8}
		\scalar{(P_c(D) - P_\infty(D)) \ue }{\ue } = \sum_{j=1}^{n+1}
		\frac{a_{1,j}}{c^{2j}} + \frac{o_c(1)}{c^{2n+2}}.
	\end{equation}

	Using the Taylor expansion
	\begin{equation*}
		\frac{1}{\sqrt{1+x}} = \sum_{j=0}^{n+1} (-1)^j \frac{1}{4^j}
		\binom{2j}{j} x^j + \mathcal{O}(|x|^{n+2}), \quad |x| < 1,
	\end{equation*}
	we have that the operator $T_c(D)$ admits the formal expansion
	\begin{equation*}
		\begin{split}
			T_c(D) &= \frac{P_c(D)}{\sqrt{-\Delta/(m^2 c^2) + 1}} \\
			&= \left( \sum_{j=0}^{n+1} \frac{(-1)^j
			\alpha_{j+1}}{m^{2j+1} c^{2j}} (-\Delta)^{j+1} \right)
			\left( \sum_{j=0}^{n+1} (-1)^j \frac{1}{(m c)^{2j} 4^j}
			\binom{2j}{j} (-\Delta)^j \right) \\
			&\qquad + \mathcal{O}\left( \frac{1}{c^{2n+4}} \right) \\
			&= P_\infty(D) + \sum_{j=1}^{2n+2} \sum_{\substack{s+t=j}}
			\frac{(-1)^j \alpha_{s+1}}{4^t m^{2j+1} c^{2j}}
			\binom{2t}{t} (-\Delta)^{j+1} + \mathcal{O}\left(
			\frac{1}{c^{2n+4}} \right).
		\end{split}
	\end{equation*}
	and hence
	\begin{equation}\label{new_q9}
		\scalar{(T_c(D) - P_\infty(D)) \ue }{\ue }
		= \sum_{j=1}^{n+1} \frac{b_{1,j}}{c^{2j}} + \frac{o_c(1)}{c^{2n+2}},
	\end{equation}
	with 
	\begin{equation*}
		b_{1,j} = \sum_{\substack{z+\ell = j \\ z \geq 1}} 
		\frac{(-1)^{z}}{m^{2z+1}} \sum_{s + t = z} \sum_{p+q = \ell} 
		\frac{\alpha_{s+1}}{4^{t}} \binom{2t}{t} 
		\scalar{(-\Delta)^{\frac{z+1}{2}} 
		\fe_{q}}{(-\Delta)^{\frac{z+1}{2}} \fe_{p}}.
	\end{equation*}
	Combining \eqref{eq:estimateec}, \eqref{new_q8}, and
	\eqref{new_q9}, we conclude
	\begin{equation*}
		\omega_c = \omega_\infty + \sum_{j=1}^{n+1} \frac{b_j}{c^{2j}} + \frac{o_c(1)}{c^{2n+2}},
	\end{equation*}
	where $b_j = -3a_j + a_{1,j} - b_{1,j}$ for $j = 1, \dots, n+1$.
\end{proof}

\begin{lemma}
	For all $j \in \mathbb{N}$ there exist unique radial functions
	$\fe_{j} \in H^{s}(\mathbb{R}^3)$ and constants $a_j, b_j \in
	\mathbb{R}$ such that for all $n \in \mathbb{N}$
	\begin{equation}\label{new_3.11}
		\left\| c^{2n}\ue  - \sum_{j=0}^{n} c^{2(n-j)}\fe_{j} \right\|_{H^s} \to 0 
		\quad \text{as } c \to +\infty,
	\end{equation}
	and
	\begin{align}
		&\left|c^{2n+2}(e_{c}-e_{\infty}) - \sum_{j=1}^{n+1}
		c^{2(n+1-j)}a_{j} \right| \to 0, \label{new_q14}\\
		&\left|c^{2n+2}(\omega_{c}-\omega_{\infty}) - \sum_{j=1}^{n+1}
		c^{2(n+1-j)}b_{j} \right| \to 0. \label{eq:estimateomegac}
	\end{align}
	In particular, $\fe_{0} = \uei $, and for $n \geq 1$,
	$\fe_{n}$ is the unique radial solution of
	\begin{equation}
		\label{eq:deffkenery}
		\Le \fe_n = - \sum_{j=0}^{n-1} P_{\infty, n-j}(D)\fe_{j} 
		- \sum_{j=0}^{n-1} b_{n-j}\fe_j
		+ \mathcal{T}_n(\fe_0, \fe_{1}, \dots, \fe_{n-1}),
	\end{equation}
	where
	\begin{equation}\label{new_3.12}
		\begin{split}
			\mathcal{T}_{n}(\fe_0, \fe_{1}, \dots, \fe_{n-1})
			&= \sum_{k=2}^{\min\{n,3\}} \frac{1}{k!}
			\sum_{\substack{i_1+\cdots+i_k=n \\ 1\leq i_1, \dots, i_k \leq n-1}}
			\mathcal{N}^{(k)}(\fe_0)[\fe_{i_1}, \dots, \fe_{i_k}].
		\end{split}      
	\end{equation}
	Here $\mathcal{T}_{n}$ is a Hartree-type nonlinearity in $\fe_0, \fe_{1}, \dots, \fe_{n-1}$, and $\mathcal{T}_{1} = 0$.
\end{lemma}

\begin{proof}
	By Lemma \ref{Lem:4.5}, Lemma \ref{newlemma:1}, and Lemma
	\ref{newlemma:2}, we have that \eqref{new_3.11}, \eqref{new_q14}
	and \eqref{eq:estimateomegac} hold for $n=1$.
	
	We now proceed by induction, following the approach in Lemma
	\ref{lem:existencefj}. 
	
	Assume that for some $n \geq 1$, there exist $\fe_{k}$ for $k = 1,
	\dots, n-1$ and $a_k$, $b_k$ for $k = 1, \dots, n$ such that
	\begin{equation}
		\label{new_q18}
		\left\| c^{2k}\ue  - \sum_{j=0}^{k} c^{2(k-j)}\fe_{j}
		\right\|_{H^s} \to 0, \quad 1 \leq k \leq n-1,
	\end{equation}
	and
	\begin{align}
		\label{new_q19}
		&e_c = e_\infty + \sum_{j=1}^{k} \frac{a_j}{c^{2j}} +
		\frac{o_c(1)}{c^{2k}},
		&&\omega_c = \omega_\infty +
		\sum_{j=1}^{k} \frac{b_j}{c^{2j}} + \frac{o_c(1)}{c^{2k}},
		&&1 \leq k \leq n,
	\end{align}
	where the $\fe_{k}$ are the unique solutions of
	\begin{equation*}
		\Le \fe_k = -\sum_{j=0}^{k-1} P_{\infty, k-j}(D)\fe_{j} -
		\sum_{j=0}^{k-1} b_{k-j}\fe_j + \mathcal{T}_k(\fe_0, \fe_{1}, \dots,
		\fe_{k-1}), \quad 1 \leq k \leq n-1.
	\end{equation*}
	
	For $2 \leq k \leq n$, we define
	\begin{equation*}
		\fe_{c,k} = c^2(\fe_{c, k-1} - \fe_{k-1}) = c^{2k} \left( \ue  -
		\sum_{j=0}^{k-1} \frac{\fe_j}{c^{2j}} \right),
	\end{equation*}
	and
	\begin{equation*}
		\Ge_{c, k} = c^2 \left( \Ge_{c, k-1} - \mathcal{T}_{k-1}(\fe_0,
		\fe_{1}, \dots, \fe_{k-2}) \right).
	\end{equation*}
	Furthermore, we let
	\begin{align*}
		&H_{c,1} = b_{c,1}\ue , && H_k = \sum_{j=0}^{k-1} b_{k-j}\fe_j,
		&& 1 \leq k \leq n.
	\end{align*}
	and
	\begin{align}
		\label{new_q15}
		&H_{c,k} = c^2(H_{c,k-1} - H_{k-1}) = c^{2k} \left( 
		\frac{H_{c,1}}{c^2} - \sum_{j=1}^{k-1} \frac{H_j}{c^{2j}} 
		\right), && 1 \leq k \leq n.
	\end{align}
	
	From \eqref{new_q11} and \eqref{4.15}, $\fe_{c,1}$ and $\fe_1$ satisfy
	\begin{align*}
		\Le \fe_{c,1} &= -c^2 P_{c,1}(D)\ue  - H_{c,1} + \Ge_{c,1}, \\
		\Le \fe_1 &= -P_{\infty,1}(D)\uei  - H_1 + \mathcal{T}_1.
	\end{align*}
	By induction, for $\fe_{c,k}$ we have that for every $k = 1, 
	\ldots, n$
	\begin{equation}\label{new_3.14}
		\Le \fe_{c, k} = c^2 \left[ -\sum_{j=0}^{k-2} c^{2k-2j-2}
		P_{c,k-j}(D)\fe_{j} - P_{c,1}(D)\fe_{c,k-1} \right] - H_{c,k}
		+ \Ge_{c,k}.
	\end{equation}
	
	Let $\fe_n$ be the unique radial solution of
	\begin{equation}\label{3.15_1}
		\Le \fe_n = -\sum_{j=0}^{n-1} P_{\infty, n-j}(D)\fe_{j} - H_n
		+ \mathcal{T}_{n}(\fe_0, \fe_{1}, \dots, \fe_{n-1}).
	\end{equation}
	
	Subtracting \eqref{new_3.14} from \eqref{3.15_1} yields
	\begin{multline}\label{new_q21}
		\Le (\fe_{c,n} - \fe_n) = c^2 \left[ -\sum_{j=0}^{n-2}
		c^{2n-2j-2} P_{c,n-j}(D)\fe_{j} - P_{c,1}(D)\fe_{c,n-1}
		\right] \\
		+ \sum_{j=0}^{n-1} P_{\infty, n-j}(D)\fe_{j} + \left(
		\Ge_{c,n} - \mathcal{T}_{n}(\fe_0, \fe_{1}, \dots, \fe_{n-1})
		\right) - (H_{c,n} - H_n).
	\end{multline}
	
	As in \eqref{3.17} and \eqref{3.18}, we have
	\begin{equation}\label{new_3.17}
		\| \Ge_{c,n} - \mathcal{T}_{n}(\fe_0, \fe_{1}, \dots,
		\fe_{n-1}) \|_{H^s} = o_c(1),
	\end{equation}
	and
	\begin{equation}\label{new_3.18}
		\left\| c^2 \left( -\sum_{j=0}^{n-2} c^{2n-2j-2}
		P_{c,n-j}(D)\fe_{j} - P_{c,1}(D)\fe_{c,n-1} \right) +
		\sum_{j=0}^{n-1} P_{\infty, n-j}(D)\fe_{j} \right\|_{H^s} =
		o_c(1).
	\end{equation}
	
	We claim that
	\begin{equation}\label{new_q20}
		\| H_{c,n} - H_n \|_{H^s} = o_c(1).
	\end{equation}
	We postpone the proof of this claim until the end.
	
	Combining Lemma \ref{Lem:3.3} with \eqref{new_q21},
	\eqref{new_3.17}, \eqref{new_3.18}, and \eqref{new_q20}, we obtain
	\begin{equation*}
		\| \fe_{c,n} - \fe_n \|_{H^s} \to 0 \quad \text{as } c \to
		\infty,
	\end{equation*}
	which is equivalent to
	\begin{equation*}
		\left\| c^{2n}\ue - \sum_{j=0}^{n} c^{2(n-j)}\fe_{j}
		\right\|_{H^s} \to 0.
	\end{equation*}
	
	By Lemma \ref{newlemma:1} and Lemma \ref{newlemma:2}, there exist
	constants $a_{n+1}, b_{n+1} \in \mathbb{R}$ such that
	\begin{equation*}
		\left| c^{2n+2}(e_{c} - e_{\infty}) - \sum_{j=1}^{n+1}
		c^{2(n+1-j)}a_{j} \right| \to 0, ~~ \left| c^{2n+2}(\omega_{c}
		- \omega_{\infty}) - \sum_{j=1}^{n+1} c^{2(n+1-j)}b_{j}
		\right| \to 0.
	\end{equation*}
	This completes the induction step and the proof of the lemma.

	\subsection*{Proof of Claim \eqref{new_q20}}

	By assumptions \eqref{new_q18} and \eqref{new_q19},
	\begin{equation}\label{new_q22}
		\begin{split}
			H_{c,1} &= b_{c,1}\ue = \left( \sum_{j=0}^{n-1}
			\frac{b_{j+1}}{c^{2j}} + \frac{o_c(1)}{c^{2n-2}} \right)
			\left( \sum_{j=0}^{n-1} \frac{\fe_{j}}{c^{2j}} +
			\frac{o_c(1)}{c^{2n-2}} \right) \\
			&= \sum_{j=0}^{n-1} \sum_{s+t=j}
			\frac{b_{t+1}\fe_s}{c^{2j}} + \frac{o_c(1)}{c^{2n-2}} \\
			&= \sum_{j=0}^{n-1} \sum_{i=0}^j
			\frac{b_{j+1-i}\fe_i}{c^{2j}} + \frac{o_c(1)}{c^{2n-2}}.
		\end{split}
	\end{equation}
	Using \eqref{new_q22} and \eqref{new_q15}, we get
	\begin{equation*}
		\begin{split}
			H_{c,n} - H_n &= c^{2n} \left( \frac{H_{c,1}}{c^2} -
			\sum_{j=1}^{n} \frac{H_j}{c^{2j}} \right) \\
			&= c^{2n} \left( \sum_{j=1}^{n} \sum_{i=0}^j
			\frac{b_{j-i}\fe_i}{c^{2j}} + \frac{o_c(1)}{c^{2n}} -
			\sum_{j=1}^{n} \sum_{i=0}^j \frac{b_{j-i}\fe_i}{c^{2j}}
			\right) \\
			&= o_c(1).
		\end{split}
	\end{equation*}
	This verifies the claim.
\end{proof}


\end{document}